\documentclass[11pt]{article}
\usepackage{amssymb,amsmath,amsthm,latexsym,amsfonts}
\usepackage{graphicx,graphics,psfrag}
\usepackage{epsf,epstopdf,url,xcolor}
\usepackage{caption}
\usepackage{subcaption}\usepackage{tikz}
\newtheorem{theorem}{Theorem}
\newtheorem{lemma}[theorem]{Lemma}

\newtheorem{corollary}[theorem]{Corollary}
\newtheorem{conjecture}[theorem]{Conjecture}

\newtheorem{claim}[theorem]{Claim}

\newtheorem{remark}[theorem]{Remark}

\numberwithin{theorem}{section}

\begin{document}
\title{\bf List 3-dynamic coloring of graphs with small maximum average degree}

\author{Seog-Jin Kim \thanks{This research was supported by Basic Science Research Program through the National Research Foundation of Korea(NRF) funded by the Ministry of Education(NRF-2015R1D1A1A01057008) (S.-J.  Kim)}\\ \\
\small Department of Mathematics Education\\[-0.8ex]
\small Konkuk University\\[-0.8ex]
\small Seoul, Korea \\
\small\tt skim12@konkuk.ac.kr \\
\and
Boram Park\thanks{This research was supported by Basic Science Research Program through the National Research Foundation of Korea(NRF) funded by the Ministry of  Science, ICT and Future Planning(2015R1C1A1A01053495) (B. Park)}\thanks{Corresponding author: borampark@ajou.ac.kr} \\ \\
\small Department of Mathematics\\[-0.8ex]
\small Ajou University\\[-0.8ex]
\small Suwon, Korea\\
\small\tt borampark@ajou.ac.kr
}
\maketitle
\begin{abstract}
An $r$-dynamic $k$-coloring of a graph $G$ is a proper $k$-coloring such that for any vertex $v$, there are at least $\min\{r,\deg_G(v) \}$ distinct colors in $N_G(v)$. The {\em $r$-dynamic chromatic number} $\chi_r^d(G)$ of a graph $G$ is the least $k$ such that there exists an $r$-dynamic $k$-coloring  of $G$.  The {\em list $r$-dynamic chromatic number} of a graph $G$ is denoted by $ch_r^d(G)$.

Recently,
Loeb et al. \cite{UI} showed that the list $3$-dynamic chromatic number of a planar graph is at most 10.
And Cheng et al. \cite{Lai-16} studied the maximum average condition to have $\chi_3^d (G) \leq 4, \ 5$, or $6$.  On the other hand,
Song et al.~\cite{SLW} showed that if $G$ is planar with girth at least 6, then $\chi_r^d(G)\le r+5$ for any $r\ge 3$.

In this paper, we study list 3-dynamic coloring in terms of maximum average degree.  We show that  $ch^d_3(G) \leq 6$ if   $mad(G) < \frac{18}{7}$,   $ch^d_3(G) \leq 7$ if  $mad(G) < \frac{14}{5}$, and $ch^d_3(G) \leq 8$ if $mad(G) < 3$.   All of the bounds are tight.
\end{abstract}

\section{Introduction}
Let $k$ be a positive integer.
A proper $k$-coloring $\phi: V(G) \rightarrow \{1, 2, \ldots, k \}$ of a graph $G$ is an assignment of colors to the vertices of $G$ so that any two adjacent vertices  receive distinct colors.
The {\em chromatic number} $\chi(G)$ of a graph $G$ is the least $k$ such that there exists a proper $k$-coloring of $G$.
An $r$-dynamic $k$-coloring of a graph $G$ is a proper $k$-coloring $\phi$  such that  for each vertex $v\in V(G)$,
either the number of distinct colors in its neighborhood is at least $r$ or the colors in its neighborhood are all distinct, that is, $|\phi(N_{G}(v))|=\min\{r,\deg_G(v)\}$.
The {\em $r$-dynamic chromatic number} $\chi_r^d(G)$ of a graph $G$ is the least $k$ such that there exists an $r$-dynamic $k$-coloring  of $G$.


A {\em list assignment} on a graph $G$ is a function
$L$ that assigns each vertex $v$ a set $L(v)$ which is
a list of available colors at $v$.
For a list assignment $L$ of a graph $G$, we say $G$ is {\em $L$-colorable}
if there exists a proper coloring $\phi$ such that $\phi(v) \in L(v)$ for every $v \in V(G)$.
A graph $G$ is said to be {\em $k$-choosable} if for any list assignment $L$ such that
$|L(v)| \geq k$ for every vertex $v$, $G$ is $L$-colorable.

For a list assignment $L$ of $G$, we say that $G$ is {\em $r$-dynamically $L$-colorable}
if there exists an  $r$-dynamic coloring  $\phi$ such that $\phi(v) \in L(v)$ for every $v \in V(G)$.
A graph $G$ is {\em  $r$-dynamically $k$-choosable} if for any list assignment $L$ with
$|L(v)| \geq k$ for every vertex $v$, $G$ is $r$-dynamically $L$-colorable.
The  {\em list $r$-dynamic chromatic number}  $ch_r^d(G)$ of a graph $G$ is the least $k$ such that $G$ is  $r$-dynamically $k$-choosable.

The notion of $r$-dynamic coloring was firstly introduced in \cite{Montgomery}, and then it was widely studied in \cite{AGS,Esperet,Sogol, Ross-Kang,KimOk,KimLeePark,KimPark}.
Note that it was also studied in \cite{Lai-16, SFCSL, SLW} with the name of {\em $r$-hued coloring}.
Similar to the Wegner's conjecture \cite{Wegner}, a conjecture about dynamic coloring of planar graphs was proposed in \cite{SFCSL}.
\begin{conjecture} \label{conj-r-dynamic}
Let $G$ be a planar graph. Then
\[
\chi_r^d(G) \le \begin{cases}
  r+3 & \text{if }1\le r \le2 \\
   r+5& \text{if }3\le r \le7 \\
  \lfloor\frac{3r}{2}\rfloor +1 & \text{if }  r \ge8. \\
\end{cases}\]
\end{conjecture}
Song, Lai, and Wu \cite{SLW} show that Conjecture \ref{conj-r-dynamic} is true for  planar graphs with girth at least 6.
\begin{theorem}[\cite{SLW}]\label{SLW-thm}
If $G$ is a planar graph with girth at least 6, $\chi_r^d(G)\le r+5$ for any $r\ge 3$.
\end{theorem}

Recently, 3-dynamic coloring has been concerned.  Loeb,  Mahoney,  Reiniger, and  Wise \cite{UI} showed that $ch_3^d (G) \leq 10$ if $G$ is a planar graph.
On the other hand, list 3-dynamic coloring was studied in \cite{Lai-16} in terms of maximum average degree, where the {\em maximum average degree} of a graph $G$, $mad(G)$, is the maximum among the average degrees of the subgraphs of $G$. It was showed in \cite{Lai-16} that
$\chi_3^d  (G) \leq 6$ if $mad(G) < \frac{12}{5}$, $\chi_3^d (G) \leq 5$ if $mad(G) < \frac{7}{3}$, and $\chi_3^d  (G) \leq 4$ if $G$ has no $C_5$-component and $mad(G) < \frac{8}{3}$.

In this paper, we study list $3$-dynamic coloring with maximum average degree condition.
For each $k \in \{6, 7, 8\}$, we study the optimal value of maximum average degree to be $ch_3^d(G) \leq k$.
First,  we give an optimal value of $mad(G)$ to be $ch_3^d(G) \leq 6$, which improves a result in \cite{Lai-16}.
\begin{theorem}\label{thm:main}
If $mad(G)<\frac{18}{7}$, then $ch_3^d(G)\le 6$.
\end{theorem}

The bound on  $mad(G)$ in Theorem~\ref{thm:main} is tight.
The graph $H$ in Figure~\ref{fig:tight} is a subcubic graph and so $ch_3^d(H)=ch(H^2)$, where the \textit{square} of $H$, denoted by $H^2$, is the graph obtained by adding to $H$ the edges connecting two vertices having a common neighbor in $H$. Note that $mad(H)=\frac{18}{7}$ and $H^2$ is isomorphic to $K_7$.
Hence we have $ch(H^2)=ch_3^d(H)=7$, which implies that the bound on $mad(G)$ in Theorem \ref{thm:main} is tight.

From the graph $H$, one can find infinitely many tight examples for Theorem \ref{thm:main}.
Given a graph $H'$ with $mad(H')\le \frac{18}{7}$ and $ch_3^d (H') \leq 7$,
let $G$ be a graph obtained from the union of two graphs $H$ and $H'$ by connecting
with internally disjoints paths of length at least five such that the end vertices in $H$ of the paths have degree two. Figure~\ref{fig:tight-many} shows a way to construct such graphs.  Note that there are at most three such paths since $H$ has three vertices of degree two.

\medskip
\noindent {\bf Remark:}
For a given graph $F$ with $mad(F) \le \frac{18}{7}$,
let $F'$ be a graph obtained by adding a path of length $\ell$ ($\ell\ge 5$) to $F$ such that the two end vertices $x$ and $y$ are in $F$ and the other internal vertices are not.  We will show that $mad(F') \le \frac{18}{7}$.

For a graph $G$, let $\rho_G$ be a function defined on the power set of $V(G)$ such that $\rho_G(A)=9|A|-7|E(G[A])|$ for any $A\subset V(G)$, {where $|A|$ denotes the number of vertices in $A$ and $|E(G[A])|$ denotes the number of edges in the subgraph induced by $A$.
Note that} $\rho_G(A)\ge 0$ for any $A\subset V(G)$ if and only if $mad(G) \le \frac{18}{7}$.

Take any subset $A'\subset V(F')$.
If $\{x,y\}\subset  A'$, then
\[\rho_{F'}(A') \geq \rho_{F'}(A'\cap V(F))+ 9(\ell-1)-7(\ell) \ge \rho_{F}(A'\cap V(F)) +2\ell-9 >\rho_{F}(A'\cap V(F))\ge 0, \]
since $\ell \geq 5$.
If  $\{x,y\}\not \subset  A'$, then
\[\rho_{F'}(A')=\rho_{F'}(A'\cap V(F))+ 9|A'-V(F)|-7|A'-V(F)| \ge \rho_{F}(A'\cap V(F))\ge0 .\]
Therefore,  $mad(F') \le \frac{18}{7}$.

\medskip

Thus, from the above Remark, it follows that $mad(G) = mad(H)=\frac{18}{7}$.
Since $\deg_G(v)=3$ for any $v\in V(H)$ and the distance (in $G$) between two vertices in $V(H)$ is at most two, all seven vertices in $V(H)$ should get distinct colors in a 3-dynamic coloring of $G$ and so $ch_3^d(G) = \chi_3^d(G) = 7$.


\begin{figure}[h!]
  \centering
\begin{tikzpicture}[thick,scale=0.9]
        \path(-1,0) coordinate (0)    (0,0) coordinate (3)  (1,0) coordinate (4)
        (0,-1) coordinate (5)    (0,1) coordinate (1)  (1,1) coordinate (2)
        (1,-1) coordinate (6);
        \path (0) edge (1) edge(4)edge(5); \path (2)   edge(1);\path (6) edge(5) edge (2);
        \fill  (0,-2) node[above]{\small $H$};
        \fill (1) circle (3pt) (2) circle (3pt) (3) circle (3pt) (4) circle (3pt)(5) circle (3pt) (6) circle (3pt)(0) circle (3pt);
        \path (2) edge[bend left]  node[above] {} (6) ;
    \end{tikzpicture}
\caption{A tight example for Theorem~\ref{thm:main}, $mad(H)=\frac{18}{7}$ and $ch^d_3(H)=7$}\label{fig:tight}
\end{figure}
\begin{figure}[h!]
 \centering
 \begin{tikzpicture}[thick,scale=0.9]
        \path(-1,0) coordinate (0)    (0,0) coordinate (3)  (1,0) coordinate (4)
        (0,-1) coordinate (5)    (0,1) coordinate (1)  (1,1) coordinate (2)
        (1,-1) coordinate (6);
        \path (0) edge (1) edge(4)edge(5); \path (2)   edge(1);\path (6) edge(5) edge (2);
        \fill (1) circle (3pt) (2) circle (3pt) (3) circle (3pt) (4) circle (3pt)(5) circle (3pt) (6) circle (3pt)(0) circle (3pt);
        \path (2) edge[bend left]  node[above] {} (6) ;
   \path (2.3,1.1) edge (5,1.1); \path (2.3,-1.1)   edge(5,-1.1);
   \path (2.3,0) edge(5,0);
    \fill (2.3,1.1) circle (3pt) (2.9,1.1) circle (3pt) (3.5,1.1) circle (3pt) (4.1,1.1) circle (3pt)
 (2.3,-1.1) circle (3pt) (2.9,-1.1) circle (3pt) (3.5,-1.1) circle (3pt) (4.1,-1.1) circle (3pt)
(2.3,0) circle (3pt) (2.9,0) circle (3pt) (3.5,0) circle (3pt) (4.1,0) circle (3pt);
 \draw [-,bend right] (1) to (2.3,1.1);
 \draw [-,bend right] (3) to (2.3,0);
 \draw [-,bend right] (5) to (2.3,-1.1);
\draw[dotted]  (0.3,0) circle (1.7);
\draw[dotted]  (6.3,0) circle (2);
\fill  (0.3,1.7) node[above] {\small $H$};
\fill  (6.4,0) node[above] {\small$H'$} (6.5,0) node[below] {\footnotesize($mad(H')\le \frac{18}{7}$)};
 \end{tikzpicture}
     \caption{Construction of a large tight example $G$ for Theorem~\ref{thm:main}, $mad(G)=\frac{18}{7}$ and $ch^d_3(G) = 7$}\label{fig:tight-many}
\end{figure}
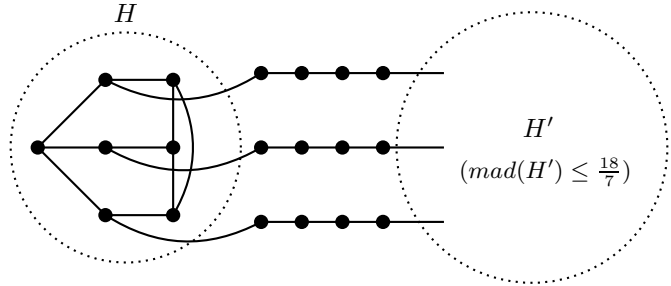

\bigskip

We also study the value of $mad(G)$ to be $ch_3^d(G)\le 7$.

\begin{theorem}\label{thm:main7}
If $mad(G)<\frac{14}{5}$, then $ch_3^d(G)\le 7$.
\end{theorem}

Let $H$ be the graph that is obtained from the Petersen graph by deleting one edge. Then $mad(H)=\frac{14}{5}$ and  $ch_3^d(H)=8$.  Thus the bound in Theorem \ref{thm:main7} is tight.

\medskip
In addition, infinitely many tight examples for Theorem \ref{thm:main7} are obtained by a similar way of construction shown in Figure~\ref{fig:tight-many}. For a given graph $H'$ with $mad(H')\le\frac{14}{5}$
and $ch_3^d(H') \leq 8$,
let $G$ be a graph obtained from the union of $H$ and $H'$ connecting by internally disjoint paths of length at least four such that $\deg_G(v)\le 3$ for any $v\in V(H)$.
Note that there are at most two such paths since $H$ has two vertices of degree two.
Then $mad(G)=\frac{14}{5}$ and $ch_3^d(G) = 8$.


\bigskip
We also show that any graph $G$ is $3$-dynamically 8-choosable if $mad(G) < 3$.

\begin{theorem}\label{thm:main8}
If $mad(G)<3$, then $ch_3^d(G)\le 8$.
\end{theorem}

The above result is also tight, since there are infinitely many tight examples.  Note that the Petersen graph $H$ satisfies $mad(H)=3$ and $\chi_3^d(H)=\chi(H^2)=10$. Now we will construct a graph $W$ with $mad(W)=3$ and $ch_3^d(W)= 9$.  Let $H_1$ and $H_2$ be the two copies of the Petersen graph.  Let $W$ be the graph obtained by connecting $H_1$ and $H_2$ by a path of length 3. Then we can check that $mad(W) = 3$ and $ch_3^d(W) = 9$.
Since $H_1$ and $H_2$ have a vertex of degree {four}, respectively, $V(H_1)$ and $V(H_2)$ do not have to have all distinct colors. Thus we have $\chi_3^d(W) = 9$ and also $ch_3^d(W) = 9$.

Similarly, we can find infinitely many tight examples for Theorem \ref{thm:main8}.
For a given graph $H'$ with $mad(H')\le 3$ and $ch_3^d(H') \leq 9$,
let $G$ be a graph obtained from the union of $H$ and $H'$ connecting by exactly one path of length at least three. Then $mad(G) = 3 = mad(H)$ and $ch_3^d(G) =  9 = ch_3^d(H)$.

\bigskip

Note that every planar graph $G$ with grith at least $g$ satisfies $mad(G)<\frac{2g}{g-2}$.  Thus
from
Theorem~\ref{thm:main},  Theorem~\ref{thm:main7}, and Theorem~\ref{thm:main8}, we have the following corollary.  Note that Theorem~\ref{thm:main8} implies Theorem~\ref{SLW-thm} when $r = 3$.

\begin{corollary} \label{corollary:main}
Let $G$ be a planar graph.  Then we have the following: \\
(1) $ch_3^d (G) \leq 6$ if the girth of $G$ is at least $9$,\\ 
(2) $ch_3^d (G) \leq 7$ if the girth of $G$ is at least $7$,\\
(3) $ch_3^d (G) \leq 8$ if the girth of $G$ is at least $6$.\\
\end{corollary}
It was showed in \cite{Havet} that $ch(G) \leq 6$ if $mad(G) < \frac{18}{7}$ and $\Delta(G) \leq 3$.  And it was also showed in \cite{Cranston-Kim-07} independently that $ch(G) \leq 6$ if $mad(G) < \frac{18}{7}$,  $\Delta(G) \leq 3$, and the girth of $G$ is at least 7.  Thus Theorem \ref{thm:main} is an extension of the results in \cite{Cranston-Kim-07, Havet}. On the other hand,  it was showed in \cite{Cranston-Kim-07} that $ch(G) \leq 7$ if $mad(G) < \frac{14}{5}$ and $\Delta(G) \leq 3$.  Thus Theorem \ref{thm:main7} is an extension of the result in \cite{Cranston-Kim-07}.  C​onsequently, Corolloary \ref{corollary:main} is an extension of the results in \cite{Cranston-Kim-07}.


\bigskip

This paper is organized as follows. In Section~2, we give preliminaries about simple reducible configurations. In Sections 3, 4, and 5, we prove Theorems~\ref{thm:main}, \ref{thm:main7}, and \ref{thm:main8}, respectively.

\section{Preliminaries}

A vertex of degree $d$ is called a $d$-\textit{vertex}, and a vertex of degree at least $d$ (at most $d$) is called a $d^+$-\textit{vertex} ($d^{-}$-\textit{vertex}).
If $x$ is adjacent to a $d$-vertex $y$ ($d^+$-vertex, or $d^{-}$-vertex), then we say that $y$ is a $d$-neighbor of $x$
($d^+$-neighbor  of $x$, or $d^{-}$-neighbor of $x$).
Two vertices $x$ and $y$ are \textit{weakly adjacent} in $G$ if they have a common 2-neighbor. In this case, we say that $x$ is a \textit{weak neighbor} of $y$.

\medskip
For each $i\in\{ 0,1,2,3\}$, we let $W_i(G)$ be the set of $3$-vertices which have exactly $i$ 2-neighbors.
That is,
\[W_i(G)=\{ v\in V(G)\mid \deg(v)=3, \text{ and exactly }i\text{ neighbors of }v\text{ are 2-vertices} \}.\]
If there is no confusion, we denote $W_i(G)$ by $W_i$. And let $[n] = \{1, 2, \ldots, n\}$.

\begin{lemma} \label{pendent:r}
Let $k\ge 6$. Let $G$ be a graph with smallest number of vertices and edges such that $ch_3^d(G)>k$.
Then the followings hold.
\begin{itemize}
\item[\rm(1)] There is no $1^-$-vertex.
\item[\rm(2)] No two $2^{-}$-vertices are adjacent.
\item[\rm(3)] No two adjacent vertices share a common 2-neighbor.
\item[\rm(4)] For each $i\in[3]$, every vertex in $W_i(G)$  has $i$ distinct weak neighbors.
\end{itemize}
\end{lemma}

\begin{figure}[h!]
\centering
\begin{subfigure}{.24\textwidth}\centering\captionsetup{width=1.\linewidth}
    \begin{tikzpicture}[thick,scale=1]
        \path (-3,0) coordinate (4)  (-2,0) coordinate (3)  (-1,0) coordinate (1)    (0,0) coordinate (2);
        \path (1)edge (2) edge (3); \path (3)edge (4);
        \path (2)edge (0.3,0) edge (0.3,0.3) edge (0.3,-0.3);
        \path (4)edge (-3.3,0) edge (-3.3,0.3) edge (-3.3,-0.3);
        \fill  (3) node[above]{$x$} (1) node[above]{$y$};
        \fill (1) circle (3pt) (3) circle (3pt)
           (2) circle (3pt) (4) circle (3pt) ;
    \end{tikzpicture}

    \bigskip
    \caption{Figure for (2)}\label{fig:C2}
    \end{subfigure}
   \qquad  \qquad
\begin{subfigure}{.24\textwidth}\centering
    \begin{tikzpicture}[thick,scale=1]
        \path (-2,0) coordinate (x)  (-1,0.5) coordinate (w)   (0,0) coordinate (y);
        \path (x)edge (y); \path (w)edge (x) edge (y);
        \path (y)edge (0.3,0) edge (0.3,0.3) edge (0.3,-0.3);
        \path (x)edge (-2.3,0) edge (-2.3,0.3) edge (-2.3,-0.3);
        \fill  (x) node[above]{$x$} (y) node[above]{$y$}(w) node[above]{$w$};
        \fill  (w) circle (3pt)
        (x) circle (3pt) (y) circle (3pt) ;
    \end{tikzpicture}
    \caption{Figure for (3)}\label{fig:ad-weakad}
    \end{subfigure}
    \qquad    \qquad  \qquad
\begin{subfigure}{.24\textwidth}\centering
    \begin{tikzpicture}[thick,scale=1]
        \path (-3,0) coordinate (1) (-2,0) coordinate (x)  (-1,0.5) coordinate (w1)   (0,0) coordinate (y)(-1,-0.5) coordinate (w2);
       \path (w1)edge (x) edge (y); \path (w2)edge (x) edge (y); \path(x) edge(1);
        \path (y)edge (0.3,0) edge (0.3,0.3) edge (0.3,-0.3);
        \path (1)edge (-3.3,0) edge (-3.3,0.3) edge (-3.3,-0.3);
        \fill  (x) node[above]{$x$} (y) node[above]{$y$}(w1) node[above]{$w_1$}(w2) node[above]{$w_2$};
        \fill  (x) circle (3pt)  (w1) circle (3pt)(w2) circle (3pt)
       (1) circle (3pt) (y) circle (3pt) ;
    \end{tikzpicture}
    \caption{Figure for (4)}\label{fig:disticnt_weak}
    \end{subfigure}\caption{Reducible configurations for Lemma~\ref{pendent:r} \label{fig:first:lemma}}
\end{figure}
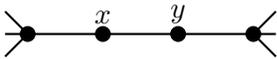
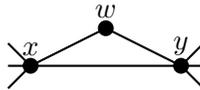
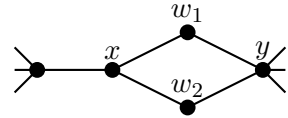

\begin{proof} We prove (1)$\sim$(4) one by one. Since $ch_3^d(G)>k$, there is a list assignment $L$ of $G$ such that $|L(v)|\ge k$ for each vertex $v$ of $G$, and $G$ is not 3-dynamically $L$-colorable.

\medskip

\noindent (1) Let $v$ be a $1^-$-vertex.
Since $H=G-\{v\}$ is smaller than $G$, $H$ is 3-dynamically $L$-colorable.
Thus there is a 3-dynamic coloring $\phi$ of $H$ such that $\phi(a)\in L(a)$ for any $a\in V(H)$.
Note that the number of available colors at $v$ is at least $k-3$.
Since $k-3\ge 1$, it is easy to see that $\phi$ can be extended to a 3-dynamic coloring  of $G$ so that $G$ is 3-dynamically $L$-colorable, which is a contradiction to the choice of $G$.

\medskip

\noindent (2) Suppose that two 2-vertices $x$ and $y$ are adjacent (See Figure~\ref{fig:first:lemma}-(a)).
Let $H=G-\{x,y\}$. Then $H$ is 3-dynamically $L$-colorable since $H$ is smaller than $G$. Thus there is a 3-dynamic coloring $\phi$ of $H$ such that $\phi(a)\in L(a)$ for any $a\in V(H)$.
Note that the number of available colors at $x$ and $y$ are at least $k-4$.
And since $k-4\ge 2$, it is easy to see that $\phi$ can be extended to a 3-dynamic coloring  of $G$ so that $G$ is 3-dynamically $L$-colorable, a contradiction to the choice of $G$.

\medskip

\noindent (3) Suppose that two adjacent vertices $x$ and $y$ share a common 2-neighbor $w$ (See Figure~\ref{fig:first:lemma}-(b)).
Let $H=G-\{w\}$.  Then $H$ is 3-dynamically $L$-colorable since $H$ is smaller than $G$.
Thus there is a 3-dynamic coloring $\phi$ of $H$ such that $\phi(a)\in L(a)$ for any $a\in V(H)$.
Note that the number of available colors at $w$ is at least $k-4$.
Since $k-4\ge 1$, it is easy to see that $\phi$ can be extended to a 3-dynamic coloring  of $G$ so that $G$ is 3-dynamically $L$-colorable, a contradiction to the choice of $G$.

\medskip

\noindent (4)  From (3),
we know that for any vertex, the set of neighbors and the set of weak neighbors are disjoint.
 Note that (4) trivially holds for the vertices in $W_1$.
Suppose that there is a vertex $x\in W_2\cup W_3$ such that $x$ has two $2$-neighbors $w_1$ and $w_2$ and the other neighbors of $w_1$ and $w_2$ are the same as a vertex $y$ (See Figure~\ref{fig:disticnt_weak}).
Let $H=G-\{w_1\}$. $H$ is 3-dynamically $L$-colorable, since $H$ is smaller than $G$.
Thus there is a 3-dynamic coloring $\phi$ of $H$ such that $\phi(a)\in L(a)$ for any $a\in V(H)$.
Note that the number of available colors at $w_1$ is at least $k-5$.
Since $k-5\ge 1$, it is easy to see that $\phi$ can be extended to a 3-dynamic  coloring  of $G$ so that $G$ is 3-dynamically $L$-colorable, a contradiction to the choice of $G$.
\end{proof}

The following are simple properties in list coloring, which will be often used in the paper.
For a function $f$ assigning a positive integer to each $v\in V(G)$,
a graph $G$ is said to be {\em $f$-choosable} if
for any list assignment $L$ such that
$|L(v)| \geq f(v)$ for every vertex $v$, $G$ is $L$-colorable.

\begin{remark}  \label{basic-lemma}\rm For each $i \in [3]$,  the graph $H_i$ in Figure~\ref{K4_e} is $f_i$-choosable.
\begin{itemize}
\item[(a)] Let $H_1=K_4-v_1v_4$ with $V(H_1)=\{v_1,v_2,v_3,v_4\}$, which is the graph in Figure \ref{K4_e}-(a).
Let $f_1(v_1)=2$, $f_1(v_2)=3$, $f_1(v_2)=2$, and $f_1(v_4)=2$.
\begin{proof}
If $L(v_1) \cap L(v_4) \neq \emptyset$, then color $v_1$ and $v_4$ with a color $c \in L(v_1) \cap L(v_4)$.
And then color $v_3$ and $v_2$.  If $L(v_1) \cap L(v_4) = \emptyset$, then color $v_2$ with a color $c \notin L(v_3)$.  And then, the number of available colors at the remaining three vertices in the path are $1$, $2$, $2$.
In each case, we can see that $H_1$ is $f_1$-choosable.
\end{proof}

\item[(b)] Let $H_2$ be a graph with $V(H_2)=\{v_1,v_2,v_3,v_4,x_1,x_2,w\}$,  which is the graph in Figure \ref{K4_e}-(b).
   Let $f_2(v_1)=f_2(v_2)=f_2(v_3)=3$, $f_2(v_4)=2$, $f_2(x_j)=4$, $f(x)=5$, $f(w)=3$.

\begin{proof}
First color the vertex $x$ with a color $c \notin L(v_1)$.  And then color the remained vertices in the order of $v_4$, $v_3$, $x_2$, $w$, $x_1$, $v_2$, $v_1$.
\end{proof}

    \item[(c)] Let $H_3$ be a graph with $V(H_3)=\{v_1,v_2,v_3,v_4,x_1,x_2,w\}$,  which is the graph in Figure \ref{K4_e}-(c).
  Let $f_3(v_1)=f_3(v_2)=f_3(v_3)=3$, $f_3(v_4)=2$, $f_3(x_j)=4$, $f(x)=5$, $f_3(w)=3$.
\begin{proof}
First color the vertex $x$ with a color $c \notin L(v_1)$.  And then color the remained vertices in the order of $v_4$, $v_3$, $x_2$, $w$, $x_1$, $v_2$, $v_1$.
\end{proof}
\end{itemize}
\end{remark}

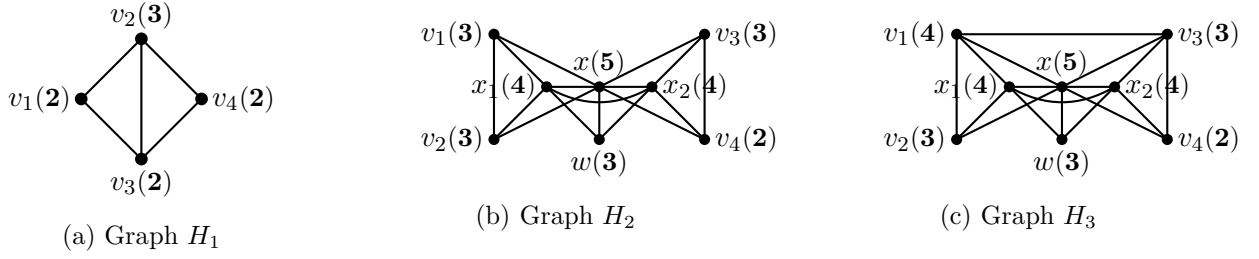
\begin{figure}
\centering
\begin{subfigure}{.24\textwidth}\centering
\begin{tikzpicture}[thick,scale=0.8]
        \path(-1,0) coordinate (1)    (0,1) coordinate (2)
        (0,-1) coordinate (3) (1,0) coordinate (4)        ;
        \path (1)edge (2) edge(3);\path (2) edge(3) edge(4);\path (3) edge(4);
        \fill  (1) node[left]{$v_1$\small$(\bf{2})$}   (2) node[above]{$v_2$\small$$\small$(\bf{3})$}    (3) node[below] {$v_3$\small$(\bf{2})$}  (4) node[right]{$v_4$\small$(\bf{2})$};
        \fill (1) circle (3pt) (2) circle (3pt) (3) circle (3pt) (4) circle (3pt);
    \end{tikzpicture}
    \caption{Graph $H_1$}
    \end{subfigure} \qquad  \qquad
    \begin{subfigure}{.24\textwidth}
\begin{tikzpicture}[thick,scale=0.7]
        \path
        (-2,1) coordinate (1)        (-2,-1) coordinate (2)        (2,1) coordinate (3)        (2,-1) coordinate (4)        (-1,0) coordinate (x1)        (1,0) coordinate (x2)       (0,0) coordinate (x) (0,-1) coordinate (w) ;
        \path (x) edge (1)edge (2) edge (3)edge (4)edge (x1) edge (x2) edge(w);        \path (x1) edge (1) edge (2) edge (w);        \path (x2) edge (3) edge (4) edge (w);   \path (1) edge (2);        \path (3) edge (4);      \path (x1) edge[bend right]  node[above] {} (x2);
        \fill        (1) node[left]{$v_1$\small$(\bf{3})$} (2) node[left]{$v_2$\small$(\bf{3})$} (3) node[right] {$v_3$\small$(\bf{3})$} (4) node[right]{$v_4$\small$(\bf{2})$}
     (x) node[above]{$x$\small$(\bf{5})$}          (x1) node[left]{$x_1$\small$(\bf{4})$}    (x2) node[right]{$x_2$\small$(\bf{4})$}      (w) node[below]{$w$\small$(\bf{3})$};
        \fill        (1) circle (3pt)        (2) circle (3pt)           (x2) circle (3pt)         (x1) circle (3pt)          (x) circle (3pt)        (w) circle (3pt)          (3) circle (3pt)        (4) circle (3pt);
    \end{tikzpicture}
    \caption{Graph $H_2$}
     \end{subfigure}%
\qquad  \qquad  \qquad
    \begin{subfigure}{.24\textwidth}
\begin{tikzpicture}[thick,scale=0.7]
        \path
        (-2,1) coordinate (1)        (-2,-1) coordinate (2)        (2,1) coordinate (3)        (2,-1) coordinate (4)        (-1,0) coordinate (x1)        (1,0) coordinate (x2)       (0,0) coordinate (x) (0,-1) coordinate (w) ;
        \path (x) edge (1)edge (2) edge (3)edge (4)edge (x1) edge (x2) edge(w);        \path (x1) edge (1) edge (2) edge (w);        \path (x2) edge (3) edge (4) edge (w);   \path (1) edge (2);        \path (3) edge (4);      \path (x1) edge[bend right]  node[above] {} (x2); \path (1) edge (3);
        \fill        (1) node[left]{$v_1$\small$(\bf{4})$} (2) node[left]{$v_2$\small$(\bf{3})$} (3) node[right] {$v_3$\small$(\bf{3})$} (4) node[right]{$v_4$\small$(\bf{2})$}
     (x) node[above]{$x$\small$(\bf{5})$}          (x1) node[left]{$x_1$\small$(\bf{4})$}    (x2) node[right]{$x_2$\small$(\bf{4})$}      (w) node[below]{$w$\small$(\bf{3})$};
        \fill        (1) circle (3pt)        (2) circle (3pt)           (x2) circle (3pt)         (x1) circle (3pt)          (x) circle (3pt)        (w) circle (3pt)          (3) circle (3pt)        (4) circle (3pt);
    \end{tikzpicture}
    \caption{Graph $H_3$}
     \end{subfigure}%

\caption{Graphs in Remark~\ref{basic-lemma}. The bold number in the parenthesis of each vertex in graph $H_i$ denotes the value of $f_i$ in Remark~\ref{basic-lemma}.}\label{K4_e}
    \end{figure}


\noindent
{\bf Notations in Figures}:
A hollow vertex in Figure \ref{C5-special} stands for a $3^+$-vertex. Throughout  following all figures, a hollow vertex always means a $3^+$-vertex, whereas the degree of a solid vertex is the number of incident edges drawn in the figure.

\begin{lemma} \label{C5-subgraph}Let $k\ge 6$.
Let $G$ be a graph with smallest number of vertices and edges such that $ch_3^d(G)>k$.
The graphs in Figure \ref{C5-special} do not appear as an induced subgraph in $G$.
\end{lemma}

\begin{figure}[h!]
\centering
\begin{subfigure}{.24\textwidth}\centering
\begin{tikzpicture}[thick,scale=1]
        \path    (0,-.5) coordinate (v4)
        (1,0.3) coordinate (v3) (0.5,1) coordinate (v2)   (-0.5,1) coordinate (v1) (-1,0.3) coordinate (v5)       (1,1.5) coordinate (2)  (-1,1.5) coordinate (w)   (-1.5,2) coordinate (u) ;
        \path (v1)edge (v2) edge(w) edge (v5);\path (v4) edge(v5) edge(v3);\path (v2) edge(2)edge(v3);\path (w)edge (u);
        \path (v4) edge (0,-0.8) edge (-0.2,-0.8) edge (0.2,-0.8);
        \path (2) edge (1.4,1.8) edge (1.4,1.5) edge (1.2,2);
        \path (u) edge (-1.8,1.8) edge (-1.8,2) edge (-1.8,2.2);
        \fill  (v1) node[left]{$v_1$}   (v2) node[above]{$v_2$}    (v3) node[below] {$v_3$}  (v4) node[right]{$v_4$} (v5) node[right]{$v_5$} (w) node[right]{$w$} (u) node[right]{$u$};
        \fill   (v1) circle (3pt) (v2) circle (3pt)
        (v3) circle (3pt)   (v5) circle (3pt)   (w) circle (3pt);
       \filldraw[fill=white!80!gray!20!] (u) circle (3pt) (v4) circle (3pt) (2) circle (3pt);
    \end{tikzpicture}
    \caption*{(a)}
    \end{subfigure} \qquad \qquad\qquad\qquad
    \begin{subfigure}{.24\textwidth}
    \centering
  \begin{tikzpicture}[thick,scale=1]
        \path(0,-1) coordinate (1)    (0,-.5) coordinate (v5)
        (1,0.2) coordinate (v4)  (1,1) coordinate (v3)   (-1,1) coordinate (v1) (0,1) coordinate (v2) (-1,0.2) coordinate (v6)        (-1.25,1.5) coordinate (w1)   (-1.5,2) coordinate (u1)  (0,2) coordinate (u2) (0,1.5) coordinate (w2)  (1.5,2) coordinate (u3) (1.25,1.5) coordinate (w3)  ;
        \path (v1)edge (v2) edge(w1) edge (v6);\path (v4) edge(v5) edge(v3);\path (v2) edge(w2)edge(v3);\path (w1)edge (u1);\path (v5) edge (v6)  ;\path (w2)edge(u2);\path (w3)edge(u3)edge (v3);
       \path (v5) edge (0,-0.9) edge (-0.3,-0.9) edge (0.3,-0.9);
        \path (u3) edge (1.8,1.8) edge (1.8,2) edge (1.8,2.2);
        \path (u2) edge (0,2.4) edge (-0.3,2.4) edge (0.3,2.4);
        \path (u) edge (-1.8,1.8) edge (-1.8,2) edge (-1.8,2.2);
        \fill  (v1) node[left]{$v_1$}   (v2) node[below]{$v_2$}    (v3) node[right] {$v_3$}  (v4) node[right]{$v_4$} (v5) node[right]{$v_5$} (w1) node[right]{$w_1$} (u1) node[right]{$u_1$}
        (w2) node[right]{$w_2$} (u2) node[right]{$u_2$} (v6) node[left] {$v_6$}
        (w3) node[right]{$w_3$} (u3) node[left]{$u_3$};
        \fill   (v1) circle (3pt) (v2) circle (3pt)(v6) circle (3pt)
        (v3) circle (3pt) (v4) circle (3pt)   (w1) circle (3pt)
        (w2) circle (3pt) (w3) circle (3pt);
      \filldraw[fill=white!80!gray!20!] (u1) circle (3pt) (u2) circle (3pt)  (u3) circle (3pt) (v5) circle (3pt);
    \end{tikzpicture}
    \caption*{(b)}
    \end{subfigure}
 \caption{Graphs in Lemma~\ref{C5-subgraph} (All labelled vertices are distinct.)}\label{C5-special}
\end{figure}
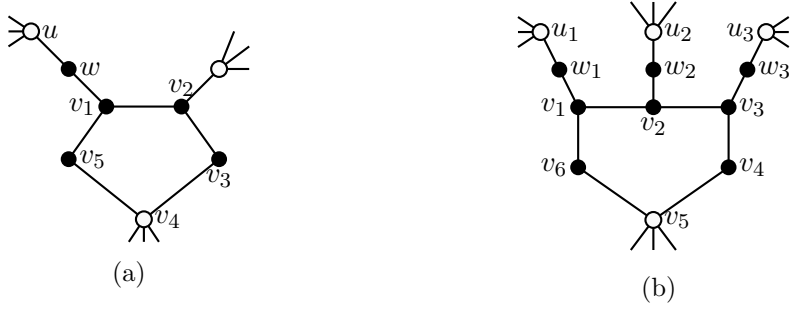
\begin{proof}
Let $L$ be a list assignment of $G$ such that $|L(v)|\ge k$ for each vertex $v$ of $G$ and $G$ is not 3-dynamically $L$-colorable.
Note that all labelled vertices in the figure are distinct. Suppose that the graph in Figure \ref{C5-special}-(a) appears in $G$ as an induced subgraph.
It has 7 vertices where $v_1, v_2$ are $3$-vertices, $v_3, v_5, w$ are $2$-vertices, and $v_4$, $u$ are $3^+$-vertex.
Let $S=\{v_1, v_3, v_5, w\}$ and $H=G-S$.
Since $H$ is smaller than $G$, $H$ is 3-dynamically $L$-colorable.
Thus there is a 3-dynamic coloring $\phi$ of $H$ such that $\phi(a)\in L(a)$ for any $a\in V(H)$, and $v_2$ and $v_4$ get distinct colors in $\phi$ (we can recolor $v_2$).

For $a\in S$, we denote by $L'(a)$ a subset of $L(a)$, which makes $\phi$ extended  to a $3$-dynamic coloring of $G$.
More precisely, $L'(a)$ is decided by the following rules.

\medskip
\noindent {\bf (Rules of deciding $L'(a)$ for $a \in S$)} \\
Let $Z = N_G(a) \cap V(H)$.  That is, $Z$ is the subset of $N_G(a)$ which are colored by the 3-dynamic coloring $\phi$.

\begin{itemize}
\item[(1)]  Remove $\phi(u)$ from $L(a)$ for each $u \in Z$.

\noindent
\item[(2)] For $u \in Z$, if $d_H(u) = 1$ and $u'$ is the neighbor of $u$ in $H$, then remove the color $\phi(u')$ from $L(a)$.

\noindent
\item[(3)] For $u \in Z$, if $d_H(u) \geq 2$, then select two colors from the neighbors of $u$, say $c_1, \ c_2$, in $L(a) \cap \{\phi(x) : x \in N_H(u)\}$ and remove $c_1$ and $c_2$ from $L(a)$.
{\noindent
\item[(4)] Let $u$ be a vertex in $H$ with $d_G(a, u) = 2$.  If $u$ and $a$ have a common 2-vertex neighbor in $G$, then remove $\phi(u)$ from $L(a)$.}
\end{itemize}

Then $L'(a)$ is the subset of $L(a)$ which are remained after (1), (2), (3), and (4).
Now we count the number of colors in $L'(a)$, and then show that $G^2[S]$ is $L'$-colorable.
(Throughout all proofs of the paper, we use a similar technique for obtaining such $L'$, we omit explanation at the other places.)

Let $c_1$ be a color which is colored at a neighbor of $v_4$ in $H$, that is, $c_1 \in \{\phi(x) : x \in N_H(v_4) \}$.  And let $c_2$ be the color which is assigned at the neighbor of $v_2$ in $H$, that is $c_2 = \phi(v_2^{'})$ where $N_G(v_2) = \{v_1, v_3, v_2^{'} \}$.
Take two colors $c_3$ and $c_4 $  from neighbors of $u$ in $H$.
We may assume the following.
\begin{eqnarray*}
&&L'(v_1)=L(v_1)-\{ \phi(v_2), \phi(v_4), \phi(u), c_2\};\\
&&L'(v_3)=L(v_3)- \{ \phi(v_2), \phi(v_4),c_1,c_2\};\\
&&L'(v_5)=L(v_5)-  \{\phi(v_2),\phi(v_4),c_1\};\\
&&L'(w) =L(w)-  \{ \phi(v_2),\phi(u), c_3,c_4\}.
\end{eqnarray*}
Therefore,
\[|L'(v_1)| \ge k-4  ,\quad |L'(v_3)| \ge k-4 ,\quad |L'(v_5)| \ge k-3 ,\quad |L'(w)| \ge k-4 .\]
Note that the subgraph of $G^2$ induced by $S$, $G^2[S]$, is isomorphic to $K_4$ minus an edge $wv_3$, a graph in Figure~\ref{K4_e}-(a).
Since $k-4\ge 2$ and $k-3\ge 3$, $G^2[S]$ is $L'$-colorable by (a) of Remark~\ref{basic-lemma}.
Then it is easy to see that $\phi$ can be extended to a 3-dynamic  coloring  of $G$ so that $G$ is 3-dynamically $L$-colorable, a contradiction.

Next, suppose that $G$ has the graph in  Figure~\ref{C5-special}-(b)  as an induced subgraph.
It has 12 vertices where $v_1$, $v_2$, $v_3$ are $3$-vertices, $v_4$, $v_6$, $w_1$, $w_2$, $w_3$ are $2$-vertices, and $v_5$, $u_1$, $u_2$, $u_3$ are $3^+$-vertex.
Let $S = \{v_1,v_2, v_3,v_4, v_6, w_1,w_2,w_3\}$.
Let $H=G-S$.
Since $H$ is smaller than $G$, $H$ is 3-dynamically $L$-colorable.
Thus there is a 3-dynamic coloring $\phi$ of $H$ such that $\phi(a)\in L(a)$ for any $a\in V(H)$.
For $a\in S$, let $L'(a)$ be a subset of $L(a)$, which makes $\phi$ extended to a $3$-dynamic coloring of $G$.
Note that $G^2[S]$ is isomorphic to the graph in  Figure~\ref{K4_e}-(c) and
\[ |L'(v_2)|\ge k-1, \quad |L'(v_i)| \ge k-2 \text{ for }i\in\{1,3,4,6\}, \quad
|L'(w_i)| \ge k-3 \text{ for }i\in\{1,2,3\}.\]
Note that we forbid just two colors at $v_4$ and $v_6$ since we will color $v_4$ and $v_6$ differently.
By (c) of Remark~\ref{basic-lemma}, it is 3-dynamically $L'$-colorable.
Thus $G$  is 3-dynamically $L$-colorable, a contradiction.
\end{proof}

\section{Proof of Theorem~\ref{thm:main}} 

In this section, we prove Theorem~\ref{thm:main}. We use the induction on the number of vertices and the number of edges.
In the following, we let $G$ be a minimal counterexample to Theorem~\ref{thm:main}. That is, $G$ is a graph with the smallest number of vertices and edges, $mad(G)<18/7$, and $ch_3^d(G)\ge 7$.
Then there exists a list assignment $L$ such that $|L(v)|\ge 6$  for each $v\in V(G)$ and $G$ is not $3$-dynamically  $L$-colorable.

\bigskip

From now on, we show that several subgraphs can not appear in $G$, which are called reducible configurations.
More precisely, we will show the following [\textbf{C1}]$\sim$[\textbf{C6}]:
\begin{itemize}
\item[] [\textbf{C1}]  There is no $1^-$ vertex. (Lemma~\ref{pendent:r}-(1))
\item[] [\textbf{C2}]  No two $2$-vertices are adjacent.  (Lemma~\ref{pendent:r}-(2))
\item[] [\textbf{C3}] No two vertices in $W_2$ are adjacent. (Lemma~\ref{no-two-W2-adjacent}, Figure~\ref{fig:C3})
\item[] [\textbf{C4}] For any vertex $x\in W_3$,
$x$ has three distinct weak neighbors in $W_1$.
(Lemma~\ref{three-nbr}, Figure~\ref{fig:C4})
\item[] [\textbf{C5}] There is no vertex in $W_1$, which has two neighbors in $W_2$. (Lemma~\ref{W1-two-W2-neighbors}, Figure~\ref{fig:C5})
\item[] [\textbf{C6}] There is no vertex $3$-vertex, which has one neighbor in $W_1$, one neighbor in $W_2$, one weak neighbor in $W_3$. (Lemma~\ref{C6:-nbr-W1W2W3}, Figure~\ref{fig:C6})
\end{itemize}

Note that [\textbf{C1}] and [\textbf{C2}] hold by Lemma~\ref{pendent:r}.
We will see that [\textbf{C3}]$\sim$[\textbf{C6}] hold.

\begin{lemma} \label{no-two-W2-adjacent}
\rm{[\textbf{C3}]} No two vertices in $W_2$ are adjacent.
\end{lemma}

\begin{figure}[h!]
 \centering
 \begin{tikzpicture}[thick,scale=0.9]
        \path (-2,1) coordinate (u1)  (-2,-1) coordinate (u2)  (3,1) coordinate (u3)    (3,-1) coordinate (u4)(-1,0.5) coordinate (v1)  (-1,-0.5) coordinate (v2)  (2,0.5) coordinate (v3)    (2,-0.5) coordinate (v4) (0,0) coordinate (x) (1,0) coordinate (y);
        \path (v1)edge (u1) edge (x); \path (v2)edge (u2) edge (x);
        \path (v3)edge (u3) edge (y); \path (v4)edge (u4) edge (y);\path (x) edge (y);
        \path (u4)edge (3.3,-1) edge (3.3,-0.7) edge (3.3,-1.3);
        \path (u3)edge (3.3,1) edge (3.3,1.3) edge (3.3,.7);
        \path (u2)edge (-2.3,-1) edge (-2.3,-0.7) edge (-2.3,-1.3);
        \path (u1)edge (-2.3,1) edge (-2.3,1.3) edge (-2.3,.7);
        \fill  (x) node[above]{$x$} (y) node[above]{$y$}(u2) node[right]{$u_2$}(u1) node[above]{$u_1$}
        (u3) node[left]{$u_3$}(u4) node[above]{$u_4$}
        (v1) node[above]{$v_1$} (v3) node[above]{$v_3$}
        (v2) node[above]{$v_2$} (v4) node[above]{$v_4$};
        \fill
        (v1) circle (3pt) (v2) circle (3pt)(v3) circle (3pt)(v4) circle (3pt) (x) circle (3pt)(y) circle (3pt);
         \filldraw[fill=white!80!gray!20!]  (u1) circle (3pt) (u2) circle (3pt)(u3) circle (3pt)(u4) circle (3pt);
    \end{tikzpicture}
    \caption{An illustration of [\textbf{C3}] (Lemma~\ref{no-two-W2-adjacent}), $x,y\in W_2$}\label{fig:C3}
    \end{figure}
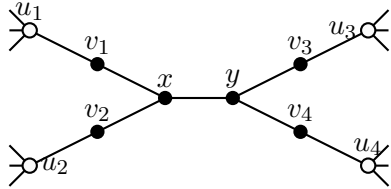

\begin{proof}
Suppose that there are two vertices $x$ and $y$ in $W_2$ that are adjacent.
That is, $x$ and $y$ are $3$-vertices and both $x$ and $y$ have exactly two $2$-neighbors.
Let $v_1, v_2$ be the $2$-neighbors of $x$ and let $v_3, v_4$ be the $2$-neighbors of $y$.
Let $u_i$ be the $3^+$-neighbor of $v_i$ for $i \in \{1, 2, 3, 4 \}$  (see Figure~\ref{fig:C3}).

By Lemma~\ref{pendent:r}-(4), $u_1\neq u_2$ and $u_3\neq u_4$.
If $u_1 = u_3$, then $x, y, v_3, u_1, v_1, v_2$ form the induced subgraph in Figure \ref{C5-special}-(a), this is impossible by Lemma \ref{C5-subgraph}.
Thus $u_1, u_2, u_3, u_4$ are all distinct.
Let $S = \{x, y, v_1, v_2, v_3, v_4 \}$ and let $H=G-S$.
Since $G$ is a minimal counterexample and $H$ is smaller than $G$, $H$ is 3-dynamically $L$-colorable.
Thus there is a 3-dynamic  coloring $\phi$ of $H$ such that $\phi(a)\in L(a)$ for any $a\in V(H)$.
For $a\in S$,  let $L'(a)$ be a subset of $L(a)$, which makes $\phi$ extended to a $3$-dynamic coloring of $G$.
Then
\[|L'(x)|\ge4, \quad\ |L'(y)| \ge 4,\quad \mbox{and} \ |L'(v_i)| \ge 3 \  \mbox{for } i \in [4].\]
Since $|L'(v_1)|+|L'(v_3)|>|L'(x)|$, we can give colors to $v_1$ and $v_3$ so that the number of available colors remained at $x$ is at least 3.
Then $G^2[\{x, y, v_2, v_4\}]$  form $K_4$ minus an edge $v_2v_4$ as in Figure~\ref{K4_e}-(a), and the numbers of available colors are $3$, $2$, $2$, $2$, respectively.  By (a) of Remark \ref{basic-lemma}, it is colorable.
This implies that $G$ is 3-dynamically $L$-colorable, a contradiction.
\end{proof}

From Lemma \ref{pendent:r}-(4), a vertex in $W_3$ has three weak neighbors.

\begin{lemma} \label{three-nbr}
\rm{[\textbf{C4}]} If $x\in W_3$ and $y$ is a weak neighbor of $x$, then $y \in W_1$.
That is, $x$ has three weak neighbors in $W_1$.
\end{lemma}

\begin{figure}[h!]
 \centering
 \begin{tikzpicture}[thick,scale=0.9]
        \path (-2,0) coordinate (u3)  (-1,0) coordinate (v3)  (0,0) coordinate (x)    (0,-0.6) coordinate (v2) (0,-1.2) coordinate (u2)  (1,0) coordinate (v1)  (2,0) coordinate (u1)    (2,-0.6) coordinate (w) (2,-1.2) coordinate (1) (3,0) coordinate (2);
        \path (x)edge (2) edge (u2) edge (u3); \path (u1)edge (1);
        \path (1)edge (2,-1.8) edge (2.3,-1.8) edge (1.7,-1.8);
        \path (u2)edge (0,-1.8) edge (0.3,-1.8) edge (-0.3,-1.8);
        \path (2)edge (3.3,0) edge (3.3,.3) edge (3.3,-0.3);
        \path (u3)edge (-2.3,0) edge (-2.3,.3) edge (-2.3,-0.3);
        \fill  (x) node[above]{$x$} (w) node[left]{$w$}
        (v1) node[above]{$v_1$} (v3) node[above]{$v_3$}(u2) node[left]{$u_2$}(u1) node[above]{$u_1$}
        (v2) node[left]{$v_2$} (u3) node[above]{$u_3$};
        \fill  (u1) circle (3pt)
        (v1) circle (3pt) (v2) circle (3pt)(v3) circle (3pt) (x) circle (3pt)(w) circle (3pt);
       \filldraw[fill=white!80!gray!20!]   (u2) circle (3pt) (u3) circle (3pt) (1) circle (3pt) (2) circle (3pt);
    \end{tikzpicture}
        \caption{An illustration of [\textbf{C4}] (Lemma~\ref{three-nbr}), $x\in W_3$, $u_1\not\in W_1$ }\label{fig:C4}
    \end{figure}
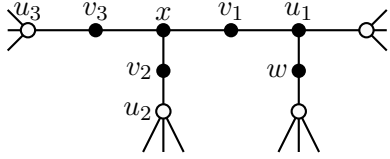

\begin{proof} Let $x$ be a $3$-vertex in $W_3$ and let $v_1, v_2, v_3$ be the $2$-neighbors of $x$.
Let $u_i$ be the other neighbor of $v_i$ for each $i\in [3]$ (so they are weak neighbors of $x$).
By [\textbf{C2}], each $u_i$ is a $3^+$-vertex.
By Lemma~\ref{pendent:r}-(4), three vertices $u_1$, $u_2$, $u_3$ are distinct.

First, we will show that $u_i$ is a 3-vertex for each $i\in [3]$.
Suppose that some $u_i$ is not a $3$-vertex for some $i\in [3]$. Without loss of generality, assume that $u_1$ is not a 3-vertex.
Then $u_1$ is a $4^+$-vertex by [\textbf{C2}].
Let $S=\{x,v_1,v_2,v_3\}$ and $H=G-S$. Then $H$ is 3-dynamically $L$-colorable,
since $G$ is a minimal counterexample and $H$ is smaller than $G$.
Thus there is a 3-dynamic coloring $\phi$ of $H$ such that $\phi(a)\in L(a)$ for any $a\in V(H)$.
For $a\in S$,  let $L'(a)$ be a subset of $L(a)$, which makes $\phi$ extended to a $3$-dynamic coloring of $G$.
Then \[|L'(v_1)| \geq 5, \ |L'(v_2)| \geq 3, \ |L'(v_3)| \geq 3 \ \mbox{ and } \ |L'(x)| \geq 3.\]
Since $G^2[S]$ is $K_4$,  $G^2[S]$ is $f_S$-choosable.  This implies that $\phi$ can be extended to  a 3-dynamic coloring of $G$  so that $G$ is 3-dynamically $L$-colorable, a contradiction.
Hence $u_i$ is a 3-vertex for each $i\in [3]$.

\medskip
Next, we will show that every $u_i$ is in $W_1$, which means that
$x$ is the only weak neighbor of $u_1$.
Without loss of generality, we may assume that $u_1$ is has another $2$-neighbor $w$ other than $v_1$.
Then all 8 vertices $x$, $v_i$'s, $u_i$'s, $w$ are distinct (See Figure~\ref{fig:C4}).
Note that $x$ and $w$ cannot have a common neighbor as all neighbors of $x$ are $v_i$'s.
Let $S = \{x, v_1,u_1,w\}$.
Since $G$ is a minimal counterexample and $H$ is smaller than $G$, $H$ is
3-dynamically  $L$-colorable.
Thus there is a 3-dynamic  coloring $\phi$ of $H$ such that $\phi(a)\in L(a)$ for any $a\in V(H)$.
{For $a\in S$,  let $L'(a)$ be a subset of $L(a)$, which makes $\phi$ extended to a $3$-dynamic coloring of $G$.}
Then  $G^2[S]$ is the graph in Figure~\ref{K4_e}-(a) and
\[  |L'(x)|\ge 2,  \  |L'(v_1)|\ge 3,  \  |L'(u_1)| \ge 2, \   |L'(w)|\ge 2.\]
By (a) of Remark~\ref{basic-lemma}, $G^2[S]$ is $L'$-colorable.
This implies that $G$ is 3-dynamically $L$-colorable, a contradiction.
\end{proof}


\begin{lemma} \label{W1-two-W2-neighbors}
\rm{[\textbf{C5}]} There is no vertex in $W_1$ which has two neighbors in $W_2$.
\end{lemma}

  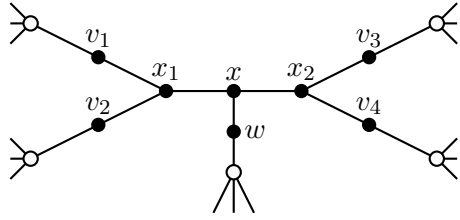
\begin{figure}[h!]\captionsetup{width=1.5\linewidth}
 \centering
  \begin{tikzpicture}[thick,scale=0.9]
        \path (-2,1) coordinate (1)  (-2,-1) coordinate (2)  (4,1) coordinate (3)    (4,-1) coordinate (4)(-1,0.5) coordinate (v1)  (-1,-0.5) coordinate (v2)  (3,0.5) coordinate (v3)    (3,-0.5) coordinate (v4) (0,0) coordinate (x1) (1,0) coordinate (x) (2,0) coordinate (x2) (1,-0.6) coordinate (w) (1,-1.2) coordinate (5);
        \path (1)  edge (x1); \path (3)edge  (x2);
        \path (2)  edge (x1); \path (4)edge (x2);\path (x1) edge (x2);\path (x) edge (5);
        \path (4)edge (4.3,-1) edge (4.3,-0.7) edge (4.3,-1.3);
        \path (3)edge (4.3,1) edge (4.3,1.3) edge (4.3,.7);
        \path (2)edge (-2.3,-1) edge (-2.3,-0.7) edge (-2.3,-1.3);
        \path (1)edge (-2.3,1) edge (-2.3,1.3) edge (-2.3,.7);
       \path (5)edge (1,-1.8) edge (1.3,-1.8) edge (.7,-1.8);
        \fill  (x) node[above]{$x$} (w) node[right]{$w$}(x1) node[above]{$x_1$}
        (v1) node[above]{$v_1$} (v3) node[above]{$v_3$}(x2) node[above]{$x_2$}
        (v2) node[above]{$v_2$} (v4) node[above]{$v_4$};
        \fill
        (v1) circle (3pt) (v2) circle (3pt)(v3) circle (3pt)(v4) circle (3pt) (x) circle (3pt)(x2) circle (3pt)(x1) circle (3pt)(w) circle (3pt);
               \filldraw[fill=white!80!gray!20!] (1) circle (3pt) (2) circle (3pt)(3) circle (3pt)(4) circle (3pt) (5) circle (3pt);
    \end{tikzpicture}
   \caption{An illustration of [\textbf{C5}] (Lemma~\ref{W1-two-W2-neighbors}), $x\in W_1$, $x_1,x_2\in W_2$ }\label{fig:C5}
    \end{figure}

\begin{proof}
Suppose that there is a vertex $x\in W_1$ such that $x$ has two neighbors $x_1 ,x_2 \in W_2$.
Then all 5 vertices in $N_G(x)\cup N_G(x_1)\cup N_G(x_2)-\{x_1,x_2\}$ are $2$-vertices.
We label those vertices as in Figure~\ref{fig:C5}.
Let $S=\{x, x_1, x_2, v_1, v_2, v_3, v_4, w \}$.
Let $u_i$ be the neighbor of $v_i$ other than $x_1$ and $x_2$ for each $i\in [4]$,
and $w'$ be the neighbor of $w$ other than $x$.
Note that  $u_1\neq u_2$ and $u_3\neq u_4$ by Lemma~\ref{pendent:r}-(4).

Then $w'\neq u_1$, otherwise the five vertices $w', w, v_1, x_1, x$ form a cycle $C_5$ and together with the three vertices $v_2$, $u_2$, $x_2$, they form the induced subgraph in Figure~\ref{C5-special}-(a).
Hence, $w'\neq u_i$ for all $i\in [4]$.
If $u_1=u_3$, then we have the graph  in Figure~\ref{C5-special}-(b), which is a contradiction.
Thus $u_1, u_2,u_3,u_4$, and $w'$ are distinct.

Let $H=G-S$.
Since $G$ is a minimal counterexample and $H$ is smaller than $G$,
$H$ is 3-dynamically  $L$-colorable.
Thus there is a 3-dynamic  coloring $\phi$ of $H$ such that $\phi(a)\in L(a)$ for any $a\in V(H)$.
{For $a\in S$,  let $L'(a)$ be a subset of $L(a)$, which makes $\phi$ extended to a $3$-dynamic coloring of $G$.}
Then
\[ |L'(x)|\ge 5,\ |L'(x_1)|\ge 4, \ |L'(x_2)|\ge 4, \ |L'(w)|\ge 3,\ |L'(v_i)|\ge3 \ (\mbox{for } i \in [4]),\]
and $G^2[S]$ is   graph $H_2$ in Figure~\ref{K4_e}-(b).
    By (b) of Remark~\ref{basic-lemma}, it is colorable.
This implies that $G$ is 3-dynamically $L$-colorable, a contradiction.
\end{proof}

\begin{lemma} \label{C6:-nbr-W1W2W3}
\rm{[\textbf{C6}]} There is no 3-vertex which has one neighbor in $W_1$, one neighbor in $W_2$, and one weak neighbor in $W_3$.
\end{lemma}

\begin{figure}[h!]
 \centering
 \begin{tikzpicture}[thick,scale=0.9]
        \path (-3,0) coordinate (1)  (-2,0) coordinate (y1)  (-1,0) coordinate (y)  (0,0) coordinate (x)    (-1,-0.6) coordinate (y2) (0,-1.4) coordinate (2) (0,-0.5) coordinate (z)   (0,-1) coordinate (z2) (0.5,-0.9) coordinate (z1)  (1,-1.3) coordinate (3)   (1,0) coordinate (v1)  (2,0) coordinate (u)    (2,-0.6) coordinate (v2) (2,-1.2) coordinate (4) (3,0) coordinate (v3) (4,0) coordinate (5);
        \path (y)edge (y2);\path (x) edge (2);\path (u) edge (4); \path (5)edge (1);\path (z1) edge (z) edge (3);
        \path (3)edge (1.3,-1.6) edge (1.5,-1.6) edge (1,-1.6);
        \path (4)edge (2,-1.8) edge (2.3,-1.8) edge (1.7,-1.8);
        \path (2)edge (0,-1.8) edge (0.3,-1.8) edge (-0.3,-1.8);
        \path (5)edge (4.3,0) edge (4.3,.3) edge (4.3,-0.3);
        \path (y2)edge  (-1,-1) edge (-1.3,-1) edge (-0.7,-1);
        \path (1)edge (-3.3,0) edge (-3.3,.3) edge (-3.3,-0.3);
        \fill  (x) node[above]{$x$} (z) node[left]{$y$}(y) node[above]{$z$}(u) node[above]{$u$}
        (z1) node[above]{$y_1$} (z2) node[left]{$y_2$}(v2) node[right]{$v_2$}(v1) node[above]{$v_1$}
        (y2) node[left]{$z_2$} (y1) node[above]{$z_1$} (v3) node[above]{$v_3$};
        \fill   (u) circle (3pt)  (z) circle (3pt)(y) circle (3pt) (y1) circle (3pt)
        (v1) circle (3pt) (v2) circle (3pt)(v3) circle (3pt) (z1) circle (3pt) (x) circle (3pt)(z2) circle (3pt);
   \filldraw[fill=white!80!gray!20!](1) circle (3pt)  (2) circle (3pt) (3) circle (3pt) (4) circle (3pt)
        (5) circle (3pt) (y2) circle (3pt);
    \end{tikzpicture}
    \caption{An illustration of [\textbf{C6}] (Lemma~\ref{C6:-nbr-W1W2W3}), $x,z\in W_1$, $y\in W_2, u\in W_3$   }\label{fig:C6}
    \end{figure}
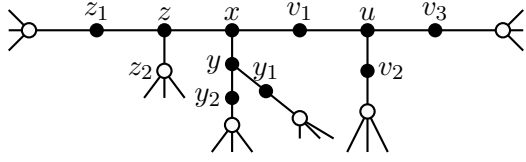

\begin{proof}
Suppose that  there exists a $3$-vertex $x$ which has one neighbor $y$ in $W_2$, one neighbor $z$ in $W_1$, one weak neighbor $u$ in $W_3$. Let $v_1$ be the 2-neighbor of $x$, let $v_2$ and $v_3$ be the other 2-neighbors of $u$.
Let $y_1$ and $y_2$ be two 2-neighbors of $y$ other than $v_1$, and let
$z_1$ and $z_2$ be two neighbors of $z$ other than $x$ (See Figure~\ref{fig:C6}).

Since $H=G-\{v_1,v_2,u\}$ is smaller than $G$, there is a 3-dynamic coloring $\phi$ of $H$  such that $\phi(a)\in L(a)$ for any $a\in V(H)$.
In the graph $H$, the vertex $x$ can be recolored without changing the colors of the other vertices, except 5 vertices $y$, $z$, $y_1$, $y_2$, $z_1$ (see  Figure \ref{recolor-lemma}) by Claim~\ref{recoloring}. (The following  claim appeared in Lemma 17 in \cite{Cranston-Kim-07}.  But, we include here for the sake of completeness.)

\begin{claim} \label{recoloring}
There is a 3-dynamic coloring $\phi'$  of $H$ such that $\phi'(a)\in L(a)$ for all $a\in V(H)$, and $\phi(x)\neq \phi'(x)$, $\phi(a)=\phi'(a)$ for any vertex $a\in V(G)\setminus\{ y, z, y_1, y_2,z_1\}$.
\end{claim}

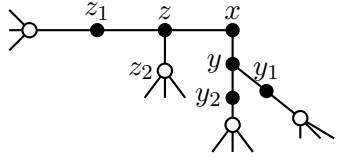
\begin{figure} \captionsetup{width=1.5\linewidth}
 \centering \begin{tikzpicture}[thick,scale=0.9]
        \path (-3,0) coordinate (1)  (-2,0) coordinate (y1)  (-1,0) coordinate (y)  (0,0) coordinate (x)    (-1,-0.6) coordinate (y2) (0,-1.4) coordinate (2) (0,-0.5) coordinate (z)   (0,-1) coordinate (z2) (0.5,-0.9) coordinate (z1)  (1,-1.3) coordinate (3)   (1,0) coordinate (v1)  (2,0) coordinate (u)    (2,-0.6) coordinate (v2) (2,-1.2) coordinate (4) (3,0) coordinate (v3) (4,0) coordinate (5);
        \path (y)edge (y2);\path (x) edge (2);  \path (x)edge (1);\path (z1) edge (z) edge (3);
        \path (3)edge (1.3,-1.6) edge (1.5,-1.6) edge (1,-1.6);
        \path (2)edge (0,-1.8) edge (0.3,-1.8) edge (-0.3,-1.8);
        \path (y2)edge  (-1,-1) edge (-1.3,-1) edge (-0.7,-1);
        \path (1)edge (-3.3,0) edge (-3.3,.3) edge (-3.3,-0.3);
        \fill  (x) node[above]{$x$} (z) node[left]{$y$}(y) node[above]{$z$}
        (z1) node[above]{$y_1$} (z2) node[left]{$y_2$}
        (y2) node[left]{$z_2$} (y1) node[above]{$z_1$}  ;
        \fill (1) circle (3pt)  (2) circle (3pt) (3) circle (3pt)
          (z) circle (3pt)(y) circle (3pt) (y1) circle (3pt)
         (z1) circle (3pt) (x) circle (3pt)(z2) circle (3pt);
           \filldraw[fill=white!80!gray!20!](1) circle (3pt)  (2) circle (3pt) (3) circle (3pt) (y2) circle (3pt);
    \end{tikzpicture}
\caption{The local structure of $H=G-\{v_1,v_2,u\}$ near the vertex $x$} \label{recolor-lemma}
    \end{figure}

\begin{proof}[Proof of Claim~\ref{recoloring}]
We uncolor the colors of 6 vertices $x$, $y$, $z$, $y_1$, $y_2$, $z_1$ from $\phi$.
Then we will show that we can recolor the vertices so that the new color of $x$ is not different from $\phi(x)$.
For $a\in S$, we denote by $L'(a)$ a subset of $L(a)$, which makes $\phi$ extended  to a $3$-dynamic coloring of $G$.
Then
\[|L'(x)|\geq 5, \ |L'(y)| \geq 4, \ |L'(z)| \geq 2, \ |L'(y_1)| \geq 3, \ |L'(y_2)| \geq 3,
\ |L'(z_1)| \geq 2. \]
Color $y$ by a color $c \notin L'(y_1)$.  Redefine $L'(v)$ by the set of available colors at $v$ after coloring $y$.  Then
\[|L'(x)|\geq 4,  \ |L'(z)| \geq 1, \ |L'(y_1)| \geq 3, \ |L'(y_2)| \geq 2,
\ |L'(z_1)| \geq 2.
\]
Color $z$ and $z_1$, then  redefine $L'(v)$ by the set of available colors at $v$ after coloring $z$ and $z_1$
\[|L'(x)|\geq 2,   \ |L'(y_1)| \geq 3, \ |L'(y_2)| \geq 2,\]
the number of available colors remained at $x$ is at least 2.  Thus we can recolor $x$ with a color distinct from $\psi(x)$.  This completes the proof of Claim \ref{recoloring}.
\end{proof}

For $a\in S$,  let $L'(a)$ be a subset of $L(a)$, which makes $\phi$ extended to a $3$-dynamic coloring of $G$.
Then
\[ |L'(v_1)| \ge 2, \ \ |L'(v_2)|\ge 2, \ \ |L'(u)|\ge 2.\]
Let $u_2$ and $u_3$ be the neighbors of $v_2$ and $v_3$ other than $u$, respectively.
Select and fix two colors $c_1$ and $c_2$ in $\{\phi(q) : q \in N_G(u_2)\setminus \{v_2\} \}$,
and then we may let
\begin{eqnarray*}
&& L'(v_1)=L(v_1)-\{ \phi(v_3),\phi(x),\phi(y),\phi(z)\};\\
&& L'(v_2)=L(v_2)-\{ \phi(v_3),\phi(u_2),c_1,c_2 \};\\
 &&L'(u)=L(u)-\{\phi(v_3),\phi(x),\phi(u_2),\phi(u_3)\}.
\end{eqnarray*}
By Claim~\ref{recoloring}, we can assume that a set of available colors at $v_2$ is not equal to that of $u$ by recoloring $x$.
As each has two available colors and all of them are not same,
we can color $v_1, v_2, u$ from the lists.
Thus  $G$ is 3-dynamically  $L$-colorable, a contradiction.
\end{proof}

\bigskip
We use discharging technique.
We define the charge of each vertex $v$ of $G$ by its degree $\deg(v)$. Note that the average charge is less than $\frac{18}{7}$. Next, we distribute their charges by the following rules, and then we show that  the new charge of each vertex is at least $\frac{18}{7}$, which leads a contradiction.

\bigskip
Recall that  $W_2$ is the set of $3$-vertices which have two 2-neighbors, and
$W_3$ is the set of $3$-vertices which have three 2-neighbors. See also Figure~\ref{fig:rules} for discharging rules.

\bigskip

\noindent {\bf Discharging Rules}

\begin{enumerate}
\item[{\bf R1.}] A $3^+$-vertex gives $\frac{2}{7}$ to  each of its 2-neighbors.

\item[{\bf R2.}] A $3^+$-vertex gives $\frac{1}{7}$ to each of its weak neighbors in $W_3$.

\item[{\bf R3.}] A $3^+$-vertex gives $\frac{1}{7}$  to each of its neighbors in $W_2$.

\item[{\bf R4.}] A $3$-vertex in $W_0$ gives $\frac{1}{7}$ to each of its neighbors $x$ in $W_1$ if $x$ has a neighbor in $W_2$ and a weak neighbor in $W_3$.

\end{enumerate}

\begin{figure}
\centering\begin{subfigure}{.4\textwidth}\centering
\begin{tikzpicture}[thick,scale=0.9]
        \path(-2,0) coordinate (4)    (0,0) coordinate (2)
        (2,0) coordinate (n);
        \path (4)edge (n);\path (4)  edge (-2.4,0)   edge(-2.4,-0.3);
        \path (n) edge (2.4,0) edge (2.4,0.3) edge(2.4,-0.3);
           \draw [->,bend right] (4) to node[below] {\small$\frac{2}{7}$}  (-0.1,-0.1);
        \fill  (4) node[above]{\tiny$3^+$-vertex $u$}   (2) node[above]{\tiny$2$-vertex};
        \fill  (2) circle (3pt) ;
          \filldraw[fill=white!80!gray!20!] (4) circle (3pt) (n) circle (3pt);
    \end{tikzpicture}
\caption{\textbf{R1} A $3^+$-vertex $u$ gives $\frac{2}{7}$ to  each of its 2-neighbors.}
    \end{subfigure} \hspace{2cm}
\begin{subfigure}{.4\textwidth}\centering
\begin{tikzpicture}[thick,scale=0.9]
        \path(-1.4,0) coordinate (4)    (-0.2,0) coordinate (2)
        (1,0) coordinate (n) (2,0) coordinate (a) (3,0) coordinate (b) (1,-0.7) coordinate (c) (1,-1.4) coordinate (d);
        \path (4)edge (b);\path (4)  edge (-1.8,0)   edge(-1.8,-0.3);
        \path (b) edge (3.4,0) edge (3.4,0.3) edge(3.4,-0.3);
    \path (d) edge (1,-1.7) edge (0.7,-1.7) edge(1.3,-1.7) edge (n);
       \draw [->,bend right] (4) to node[below] {\small$\frac{1}{7}$}  (0.9,-0.1);
       \fill  (4) node[above]{\tiny$3^+$-vertex $u$}   (n) node[above]{\tiny in $W_3$};
        \fill (2) circle (3pt) (n) circle (3pt) (a) circle (3pt)  (c) circle (3pt);
        \filldraw[fill=white!80!gray!20!] (4) circle (3pt) (b) circle (3pt)  (d) circle (3pt);
    \end{tikzpicture}
      \caption{\textbf{R2} A $3^+$-vertex $u$ gives $\frac{1}{7}$ to each of its weak neighbors in $W_3$.}
    \end{subfigure}
    \vspace{1cm}

    \begin{subfigure}{.4\textwidth}\centering
\begin{tikzpicture}[thick,scale=0.9]
        \path(-1.5,0) coordinate (4)    (0,0) coordinate (2)
        (0,0) coordinate (n) (1,0) coordinate (a) (2,0) coordinate (b) (0,-1.3) coordinate (d)
        ;
        \path (4)edge (b);\path (4)  edge (-1.9,0)   edge(-1.8,-0.3);
        \path (b) edge (2.4,0) edge (2.4,0.3) edge(2.4,-0.3);
    \path (d) edge (-0.3,-1.8) edge (0,-1.8) edge(0.3,-1.8) edge (n);
       \draw [->,bend right] (4) to node[below] {\small$\frac{1}{7}$}  (-.1,-0.1);
        \fill  (4) node[above]{\tiny$3^+$-vertex $u$}   (n) node[above]{\tiny in $W_2$};
        \fill   (2) circle (3pt) (n) circle (3pt) (a) circle (3pt)   (0,-0.7) circle (3pt);
             \filldraw[fill=white!80!gray!20!] (4) circle (3pt) (b) circle (3pt)   (d) circle (3pt);
    \end{tikzpicture}
     \caption{\textbf{R3}  A $3^+$-vertex $u$ gives $\frac{1}{7}$  to each of its neighbors in $W_2$.}
     \end{subfigure}%
\hspace{2cm}
    \begin{subfigure}{.4\textwidth}\centering
\begin{tikzpicture}[thick,scale=0.9]
        \path(-1.4,0) coordinate (4)    (0,0) coordinate (2)
        (1,0) coordinate (n)  (2,0) coordinate (b) (0,-0.7) coordinate (d);
        \path (4)edge (b);\path (4)  edge (-1.8,0)   edge(-1.7,-0.3);
        \path (b) edge (2.4,0)   edge(2.4,-0.3);
    \path (d) edge (0,-1.8) edge(0.8,-1.8) edge (2);
    \path (0,-1.5) edge (-0.2,-1.8) edge (0.2,-1.8);
    \path (0.6,-1.5) edge (0.6,-1.8) edge (1,-1.8);
       \draw [->,bend right] (4) to node[below] {\small$\frac{1}{7}$}  (-0.1,-0.1);
        \fill  (4) node[above]{\tiny $u\in W_0$}   (2) node[above]{\tiny $x\in W_1$}
         (d) node[right]{\tiny in $W_2$}  (1.9,0) node[above]{\tiny in $W_3$};
        \fill  (2) circle (3pt) (n) circle (3pt)  (b) circle (3pt) (d) circle (3pt)
        (0,-1.2) circle (3pt)    (0.35,-1.2) circle (3pt)
        ;
                \filldraw[fill=white!80!gray!20!] (4) circle (3pt) (0,-1.5) circle (3pt)   (0.6,-1.5) circle (3pt)  ;
    \end{tikzpicture}
    \caption{\textbf{R4} A $3$-vertex $u$ in $W_0$ gives $\frac{1}{7}$ to each of its neighbors $x$ in $W_1$ if $x$ has a neighbor in $W_2$ and a weak neighbor in $W_3$.}
     \end{subfigure}%
   \caption{An illustration of Discharging Rules}\label{fig:rules}
    \end{figure}
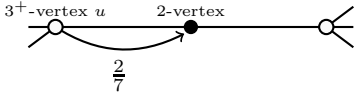
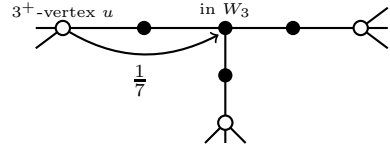
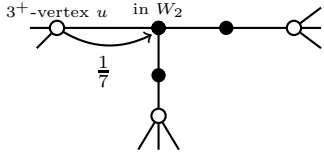
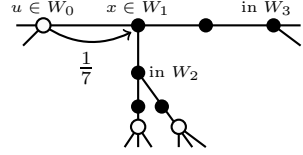

Let $d^*(u)$ be the new charge after discharging.  We will show  that $d^*(u) \geq  \frac{18}{7}$ for all $u\in V(G)$.
Note that by
[\textbf{C1}] each vertex of $G$ is a $2^+$-vertex.

\bigskip
\noindent (1)  Suppose that $\deg(u) =2$. By [\textbf{C2}]  the neighbors of $u$ are $3^+$-vertices and so it receives $\frac{2}{7}$ from each of its neighbors by {\bf R1}, and so  the new charge of $u$ is
\[ d^*(u) = 2 +\frac{2}{7}+\frac{2}{7} =\frac{18}{7}.\]

\bigskip

\noindent(2) Suppose that $\deg(u) =3$.
If $u\in W_0$, then $u$ does not have a 2-neighbor and a weak neighbor, and so $u$ might give charge $\frac{1}{7}$ to each of neighbors by {\bf R3} and {\bf R4}. Thus  the new charge of $u$  satisfies
\[ d^*(u) \ge  3 - 3 \times \frac{1}{7} =\frac{18}{7}.\]
Next, suppose that $u\in W_1\cup W_2\cup W_3$. Then $u$ gives $\frac{2}{7}$ to each of its 2-neighbors by {\bf R1}.

\medskip
\smallskip\noindent (2-1). Suppose that $u \in W_3$. 

By Lemma~\ref{three-nbr} ([\textbf{C4}]), $u$ has three distinct weak neighbors $x_1$, $x_2$, $x_3$ in $W_1$.
Then
$u$ receives $\frac{1}{7}$ from each $x_i$ by {\bf R2}.
Since each $x_i$ is not in $W_2\cup W_3$ and so $u$ does not give any charge to them by {\bf R2} or {\bf R3}.
Thus the new charge of $u$ satisfies
\[ d^*(u) \ge  3 - 3 \times \frac{2}{7} + 3 \times\frac{1}{7} = 3-\frac{3}{7}= \frac{18}{7}.\]

\smallskip\noindent (2-2). Suppose that $u\in W_2$.

Let $x$ be the $3^+$-neighbor of $u$, and $x_1$ and $x_2$ be the two weak neighbors of $u$. See Figure~\ref{fig:2-2} for an illustration.
\begin{figure}\centering
\begin{tikzpicture}[thick,scale=0.9]
        \path(-1.5,0) coordinate (4)    (0,0) coordinate (2)
        (0,0) coordinate (n) (1,0) coordinate (a) (2,0) coordinate (b) (0,-1.3) coordinate (d)
        ;
        \path (4)edge (b);\path (4)  edge (-1.9,0)   edge(-1.8,-0.3);
        \path (b) edge (2.4,0) edge (2.4,0.3) edge(2.4,-0.3);
    \path (d) edge (-0.3,-1.8) edge (0,-1.8) edge(0.3,-1.8) edge (n);
        \fill  (4) node[above]{\small $x$}   (n) node[above]{\small $u$} (b) node[above]{\small $x_1$} (d) node[right]{\small $x_2$};
        \fill (2) circle (3pt) (n) circle (3pt) (a) circle (3pt) (b) circle (3pt)   (d) circle (3pt)  (0,-0.7) circle (3pt);
              \filldraw[fill=white!80!gray!20!]  (4) circle (3pt) (b) circle (3pt)  (d) circle (3pt);
    \end{tikzpicture}
     \caption{An illustration of case (2-2), $u\in W_2$}\label{fig:2-2}
     \end{figure}
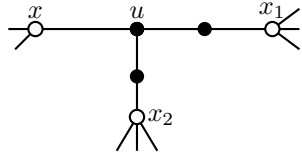%
Then $u$ receives $\frac{1}{7}$ from  $x$ by {\bf R3}.
By [\textbf{C4}], each $x_i$ is not in $W_3$.
Thus $u$ does not give any charge to them by {\bf R2}.
By [\textbf{C3}], $x\not\in W_2$ and so $u$ does not send any charge to $x$ by {\bf R3}.
Since $u\not\in W_0$, $u$ does not send any charge to $x$ by {\bf R4}.
Thus, in total,
\[ d^*(u) \ge  3 - 2 \times \frac{2}{7} + \frac{1}{7} = 3-\frac{3}{7} =\frac{18}{7}.\]

\smallskip\noindent (2-3). Suppose that $u\in W_1$.

Let $x$ and $z$ be the two $3^+$-neighbors and $w$ be the weak neighbor of $u$.
Then  by [\textbf{C5}], we may assume that $z \not\in W_2$.
Thus  $u$  gives at most $\frac{1}{7}$  to $x$ and $z$ in total by {\bf R3}.
By {\bf R2}, $u$ gives at most $\frac{1}{7}$ to $w$.

\smallskip\noindent (2-3-1). Suppose that $x \not\in W_2$ or $w\not\in W_3$.
If $x \not \in W_2$, then $u$ does not give any charge to $x$ by {\bf R3}.
If $w \not \in W_3$, then $u$ does not give any charge to $w$ by {\bf R2}.
Thus $u$ gives at most $\frac{1}{7}$ to $x$, $z$, and $w$ in total,
 \[ d^*(u) \ge  3 - \frac{2}{7} - \frac{1}{7}   = \frac{18}{7}.\]

    \begin{figure}\centering
\begin{tikzpicture}[thick,scale=0.9]
        \path(-1.5,0) coordinate (4)    (0,0) coordinate (2)
        (1,0) coordinate (n)  (2,0) coordinate (b) (0,-0.7) coordinate (d);
        \path (4)edge (4.3,0); \path (b) edge (2,-1.7);
        \path (4)  edge (-1.9,0)   edge(-1.8,-0.3);
        \path (4,0) edge(4.3,0.3) edge(4.3,-0.3);
    \path (d) edge (0,-1.8) edge(0.8,-1.8) edge (2);
    \path (0,-1.5) edge (-0.2,-1.8) edge (0.2,-1.8);
    \path (0.6,-1.5) edge (0.6,-1.8) edge (1,-1.8);
    \path (2,-1.4) edge (1.7,-1.7) edge (2.3,-1.7);
        \fill  (4) node[above]{\small $z$}   (2) node[above]{\small $u$} (b) node[above]{\small $w$}
         (d) node[right]{\small $x$};
        \fill (4) circle (3pt) (2) circle (3pt) (n) circle (3pt)  (b) circle (3pt) (d) circle (3pt)
        (0,-1.2) circle (3pt)    (0.35,-1.2) circle (3pt)
        (3,0) circle (3pt)  (2,-0.7) circle (3pt)  ;
    \filldraw[fill=white!80!gray!20!]  (0,-1.5) circle (3pt)  (0.6,-1.5) circle (3pt) (4,0) circle (3pt)  (2,-1.4) circle (3pt);
    \end{tikzpicture}
    \caption{An illustration for case (2-3-2), $u\in W_1$, $x\in W_2$,  and $w\in W_3$.}\label{fig:2-3}
     \end{figure}
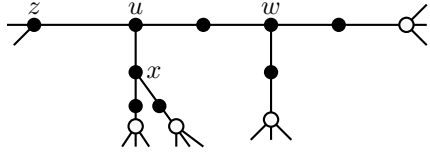%

\smallskip\noindent (2-3-2). Suppose that $x \in W_2$ and $w\in W_3$. See Figure~\ref{fig:2-3} for an illustration.
Then $u$ gives $\frac{1}{7}$  to $w$ by {\bf R2}.
By [\textbf{C6}], $z\not\in W_1$, which implies that $z\in W_0$.
Then  $u$ receives $\frac{1}{7}$ from $z$ and $u$ does not send any charge to $z$ by {\bf R4}.
(Note that $u$ is a vertex in $W_1$, which has a neighbor in $W_2$ and  one weak neighbor in $W_3$.)
Thus
\[ d^*(u) \ge  3 - \frac{2}{7} - 2 \cdot \frac{1}{7}  + \frac{1}{7}  = \frac{18}{7}.\]

\bigskip
\noindent (3) Suppose that $\deg(u) \ge 4$.

In this case, $u$ gives charge at most $\frac{2}{7}$ to its neighbors by {\bf R1}, {\bf R2} and {\bf R3}.
Note that any weak neighbor of $u$ is not in $W_3$ by [\textbf{C4}] and so $u$ does not give any charge to its weak neighbor by {\bf R2}.
Thus
\[ d^*(u) \ge  \deg(u) -\deg(u)\times \frac{2}{7}=\frac{5}{7}\deg(u) >\frac{18}{7}.\]

This completes the proof of Theorem \ref{thm:main}.


\section{Proof of Theorem~\ref{thm:main7}}

 We use the induction on the number of vertices and the number of edges.
In the following, we let $G$ be a minimal counterexample to Theorem \ref{thm:main7}. That is, $G$ is a graph with the smallest number of vertices and edges, $mad(G)<14/5$, and $ch_3^d(G)\ge 8$.
Then there exists a list assignment $L$ such that $|L(v)|\ge 7$  for each $v\in V(G)$ and $G$ is not $3$-dynamically $L$-colorable.

\begin{lemma} \label{4-2-nbr7} For $k\in\{3,4\}$, any $k$-vertex has at most $(k-2)$ $2$-neighbors.
\end{lemma}

\begin{proof} Let $k\in\{3,4\}$ and let
$v$ be a $k$-vertex, and $v_1$, $v_2$, \ldots, $v_k$ be its neighbors.
Suppose that $v$ has at least $(k-1)$ 2-neighbors $v_1$, \ldots, $v_{k-1}$.
Let $H=G-vv_k$.
Then $mad(H)<14/5$.
Since $G$ is a minimal counterexample and $H$ is smaller than $G$, $H$ is
3-dynamically $L$-colorable.
Thus there is a 3-dynamic  coloring $\phi$ of $H$ such that $\phi(a)\in L(a)$ for any $a\in V(H)$.
Then uncolor the vertex $v$ and its 2-neighbors $v_1$, \ldots, $v_{k-1}$.

Note that the number of forbidden colors at $v$ is at most $(k-1)+3=k+2\le 6$.
Thus $v$ has at least one available color.
We color $v$ first with an available color.
Then we recolor each 2-neighbor of $v$ one by one.
Since the number of available colors at each 2-neighbor of $v$ is two, and so they are colorable so that $v$ has three distinct colored neighbors. Thus $G$ is 3-dynamically $L$-colorable, a contradiction.
\end{proof}

\begin{lemma} \label{three-W1-nbr7} No two $3$-vertices $x$ and $y$ in $W_1$ are adjacent.
\end{lemma}

\begin{proof}
Let $x$ and $y$ be $3$-vertices such that $x,y\in W_1$ and $xy\in E(G)$.
Let $x'$ and $y'$ be  2-neighbors of $x$ and $y$, respectively.
And let $w_1$ and $w_2$ be the other neighbor of $x'$ and $y'$, respectively.
See Figure~\ref{fig:4-2}.
Let $H=G-\{x',y'\}$.
Then $mad(H)<14/5$.
Since $G$ is a minimal counterexample and $H$ is smaller than $G$, $H$ is
3-dynamically  $L$-colorable.
Thus there is a 3-dynamic  coloring $\phi$ of $H$ such that $\phi(a)\in L(a)$ for any $a\in V(H)$.
Then uncolor the colors of $x$ and $y$.
Then the number of available colors at  $x$ is at least 3, and that of $y$ is also at least 3.
Color $x$ with a color which is different from the color assigned at $w_1$, and $y$ with a color which is different from the color assigned at $w_2$.  Let $L'(x')$ and $L'(y')$ be the set of available colors at $x'$ and $y'$, respectively.

\begin{figure}
\centering\begin{subfigure}{.5\textwidth}\centering
\begin{tikzpicture}[thick,scale=0.9]
        \path(-2,0) coordinate (w1)   (0.5,1.6) coordinate (kk)   (-1,0) coordinate (x')
        (0,0) coordinate (x)  (1,0) coordinate (y) (2,0) coordinate (y') (3,0) coordinate (w2) (0,-1) coordinate (xn) (1,-1) coordinate (yn) ;
        \path (-2.3,0) edge (3.3,0); \path (x) edge (0,-1.3);
        \path (y)  edge (1,-1.3);
        \path (w1) edge(-2.3,0.3) edge(-2.3,-0.3);
    \path (w2) edge(3.3,0.3) edge(3.3,-0.3);
    \path (xn) edge (-0.3,-1.3) edge (0.3,-1.3);
    \path (yn) edge (0.7,-1.3) edge (1.3,-1.3);
        \fill  (w1) node[above]{\small $w_1$}
        (w2) node[above]{\small $w_2$}
        (x) node[above]{\small $x$}
        (y) node[above]{\small $y$}
        (x') node[above]{\small $x'$}
        (y') node[above]{\small $y'$};
        \fill (w1) circle (3pt) (w2) circle (3pt) (x) circle (3pt)  (y) circle (3pt) (x') circle (3pt)
        (y') circle (3pt)  ;
    \filldraw[fill=white!80!gray!20!]  (w1) circle (3pt)  (w2) circle (3pt) (xn) circle (3pt)  (yn) circle (3pt);
    \end{tikzpicture}\caption{Case when $w_1\neq w_2$}\end{subfigure}
    \begin{subfigure}{.4\textwidth}\centering
\begin{tikzpicture}[thick,scale=0.9]
        \path(0.5,1.6) coordinate (w1)    (-0.3,1) coordinate (x')
        (0,0) coordinate (x)  (1,0) coordinate (y) (1.3,1) coordinate (y')  (0,-0.7) coordinate (xn) (1,-0.7) coordinate (yn);
        \path (x) edge (y); \path (x) edge (0,-1);
        \path (y)  edge (1,-1); \path (x') edge (x) edge(w1);\path (y') edge (y) edge(w1);
        \path (w1) edge(0.5,1.9) edge(0.2,1.8) edge(0.7,1.8);
    \path (xn) edge (-0.3,-1) edge (0.3,-1);
    \path (yn) edge (0.7,-1) edge (1.3,-1);
        \fill  (w1) node[right]{\small $w_1$}
        (x) node[left]{\small $x$}
        (y) node[right]{\small $y$}
        (x') node[left]{\small $x'$}
        (y') node[right]{\small $y'$};
        \fill  (x) circle (3pt)  (y) circle (3pt) (x') circle (3pt)
        (y') circle (3pt)   ;
        \filldraw[fill=white!80!gray!20!]  (w1) circle (3pt) (xn) circle (3pt)  (yn) circle (3pt) ;
    \end{tikzpicture}\caption{Case when $w_1=w_2$}\end{subfigure}
    \caption{An illustration for Lemma~\ref{three-W1-nbr7}.}\label{fig:4-2}
     \end{figure}
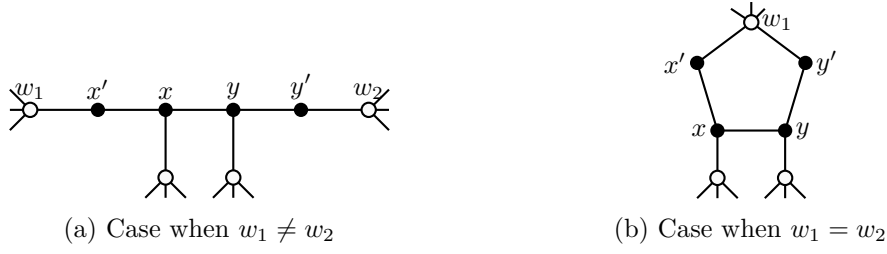%

\medskip
Now, we consider two cases.

\noindent
Case 1: $w_1  \neq w_2$ (See Figure~\ref{fig:4-2}-(a)).

Since $|L'(x')| \geq 1 $ and $|L'(y')| \geq 1 $, we can color $x'$ and $y'$ to have a dynamic 3-coloring.

\medskip
\noindent
Case 2: $w_1  = w_2$ (See Figure~\ref{fig:4-2}-(b)).

If the degree of $w_1$ in $H$ is at least three, then $x'$ and $y'$ do not have to use different color and so we have a 3-dynamic coloring.  Next, if the degree of $w_1$ is 2 in $H$, then $|L'(x')| \geq 2 $ and $|L'(y')| \geq 2 $.  So they are colorable. Thus $G$ is 3-dynamically $L$-colorable, a contradiction.
\end{proof}

\begin{lemma} \label{three-W1-2-nbr7} No $3$-vertex has three neighbors in $W_1$.
\end{lemma}

\begin{proof}
  Suppose that there is a vertex $x$ having three neighbors $x_1,x_2,x_3$ in $W_1$.
  Let $x'_i$ be the 2-neighbor of $x_i$ for each $i\in [3]$. See Figure~\ref{fig:4-3} for an illustration.
  Let $H=G-\{ x,x_1,x_2,x_3,x'_1,x'_2,x'_3\}$.
  Since $G$ is a minimal counterexample and $H$ is smaller than $G$,
  $H$ is 3-dynamically $L$-colorable.
Thus there is a 3-dynamic  coloring $\phi$ of $H$ such that $\phi(a)\in L(a)$ for any $a\in V(H)$.
Then the number of available colors at $x$ is at least $4$,
that of $x_i$ is at least $3$ for each $i\in [3]$.
We give a color to $x_1$, $x_2$, $x_3$, $x$ with their available colors so that they get distinct colors.
Then in the resulting coloring, the number of available colors at $x'_i$ is  at least $1$.
We color $x'_1$, $x'_2$, $x'_3$ by that available colors.
Here, the only thing that we have to concern is the case where $x'_i$ and $x'_j$ share a common neighbor and they get the same color.
\begin{figure}
\centering
\begin{subfigure}{.3\textwidth}\centering
\begin{tikzpicture}[thick,scale=0.8]
        \path(-2,0) coordinate (1)    (-1,0) coordinate (x3)   (0.3,1.8) coordinate (kk)
        (0,0) coordinate (x)  (1,0) coordinate (x1) (0,-1) coordinate (x2)
        (2,0) coordinate (5) (-1.7,0.6) coordinate (x3')
        (-2.4,1.2) coordinate (2) (1.7,0.6) coordinate (x1') (2.4,1.2) coordinate (6)
        (-0.7,-1.5) coordinate (x2') (-1.4,-2) coordinate (3) (0.7,-1.5) coordinate (4);
       \path (-2.3,0) edge (2.3,0); \path (x3) edge (2);\path (x1) edge (6);
      \path (x2)  edge (3) edge (4) edge (x);
     \path (1) edge(-2.3,0.3) edge(-2.3,-0.3);
    \path (5) edge(2.3,0.3) edge(2.3,-0.3);
   \path (2) edge (-2.6,1.4) edge (-2.6,1);
   \path (6) edge (2.6,1.4) edge (2.6,1);
   \path (3) edge (-1.7,-2.2) edge (-1.1,-2.2);
   \path (4) edge (1,-1.7) edge (0.4,-1.7);
        \fill  (x2) node[right]{\small $x_3$}
        (x3) node[above]{\small $x_2$}
        (x) node[above]{\small $x$}
        (x1) node[above]{\small $x_1$}
        (x1') node[above]{\small $x'_1$}
        (x2') node[above]{\small $x'_3$}
                (x3') node[above]{\small $x'_2$};
        \fill (x) circle (3pt) (x1) circle (3pt) (x1') circle (3pt)
        (x2) circle (3pt) (x2') circle (3pt)
        (x3) circle (3pt) (x3') circle (3pt) ;
            \filldraw[fill=white!80!gray!20!] (1) circle (3pt) (2) circle (3pt) (3) circle (3pt) (4) circle (3pt) (5) circle (3pt) (6) circle (3pt);
    \end{tikzpicture}\caption{Case when $x'_1$, $x'_2$ and $x'_3$ do not share a neighbor}\end{subfigure} \quad
\begin{subfigure}{.3\textwidth}\centering
\begin{tikzpicture}[thick,scale=0.8]
        \path(-2,0) coordinate (1)    (-1,0) coordinate (x3)
        (0,0) coordinate (x)  (1,0) coordinate (x1) (0,0.7) coordinate (x2)
        (2,0) coordinate (5) (-1.7,0.6) coordinate (x3')
        (0,2.1) coordinate (2) (1.7,0.6) coordinate (x1')
        (0,1.4) coordinate (x2')   (0.5,1.2) coordinate (4) (-1.4,-1.8) coordinate (k);
       \path (-2.3,0) edge (2.3,0); \path (x3') edge (2) edge (x3);\path (x1') edge (2) edge (x1);
      \path (x2')  edge (2)  edge (x) edge (x2); \path (x2) edge (4);
     \path (1) edge(-2.3,0.3) edge(-2.3,-0.3);
    \path (5) edge(2.3,0.3) edge(2.3,-0.3);
   \path (2) edge (-0.3,2.3) edge (0.3,2.3);
   \path (4) edge (0.7,1.3) edge (0.7,1);
        \fill  (x2) node[right]{\small $x_3$}
        (x3) node[above]{\small $x_2$}
        (x) node[below]{\small $x$}
        (x1) node[above]{\small $x_1$}
        (x1') node[above]{\small $x'_1$}
        (x2') node[left]{\small $x'_3$}  (2) node[right]{\small $w$}
                (x3') node[above]{\small $x'_2$};
        \fill (x) circle (3pt) (x1) circle (3pt) (x1') circle (3pt)
        (x2) circle (3pt) (x2') circle (3pt)
        (x3) circle (3pt) (x3') circle (3pt) ;
       \filldraw[fill=white!80!gray!20!] (1) circle (3pt) (2) circle (3pt)   (4) circle (3pt) (5) circle (3pt);
        \end{tikzpicture}
        \caption{Case when $x'_1$, $x'_2$ and $x'_3$ share a neighbor $w$}\end{subfigure}\quad
    \begin{subfigure}{.3\textwidth}\centering
\begin{tikzpicture}[thick,scale=0.8]
             \path(-2,0) coordinate (1)    (-1,0) coordinate (x3)
        (0,0) coordinate (x)  (1,0) coordinate (x1) (0,-1) coordinate (x2)
        (2,0) coordinate (5) (-1.7,0.6) coordinate (x3')
        (0,1.5) coordinate (2) (1.7,0.6) coordinate (x1')
        (-0.7,-1.5) coordinate (x2') (-1.4,-2) coordinate (3) (0.7,-1.5) coordinate (4);
       \path (-2.3,0) edge (2.3,0); \path (x3') edge (2) edge (x3);\path (x1') edge (2) edge (x1);
      \path (x2')  edge (3)  edge (x2); \path (x2) edge (x) edge (4);
     \path (1) edge(-2.3,0.3) edge(-2.3,-0.3);
    \path (5) edge(2.3,0.3) edge(2.3,-0.3);
   \path (2) edge (-0.3,1.8) edge (0.3,1.8);
   \path (3) edge (-1.7,-2.2) edge (-1.1,-2.2);
   \path (4) edge (1,-1.7) edge (0.4,-1.7);
        \fill  (x2) node[right]{\small $x_3$}
        (x3) node[above]{\small $x_2$}
        (x) node[above]{\small $x$}
        (x1) node[above]{\small $x_1$}
        (x1') node[above]{\small $x'_1$}
        (x2') node[above]{\small $x'_3$}  (2) node[right]{\small $w$}
                (x3') node[above]{\small $x'_2$};
        \fill (x) circle (3pt) (x1) circle (3pt) (x1') circle (3pt)
        (x2) circle (3pt) (x2') circle (3pt)
        (x3) circle (3pt) (x3') circle (3pt) ;
        \filldraw[fill=white!80!gray!20!](1) circle (3pt) (2) circle (3pt) (3) circle (3pt)   (4) circle (3pt) (5) circle (3pt);
        \end{tikzpicture}
 \caption{Case when $x'_1$ and $x'_2$ share a neighbor $w$ but $x'_3$ does not}\end{subfigure}
    \caption{An illustration for Lemma~\ref{three-W1-2-nbr7}.}\label{fig:4-3}
     \end{figure}
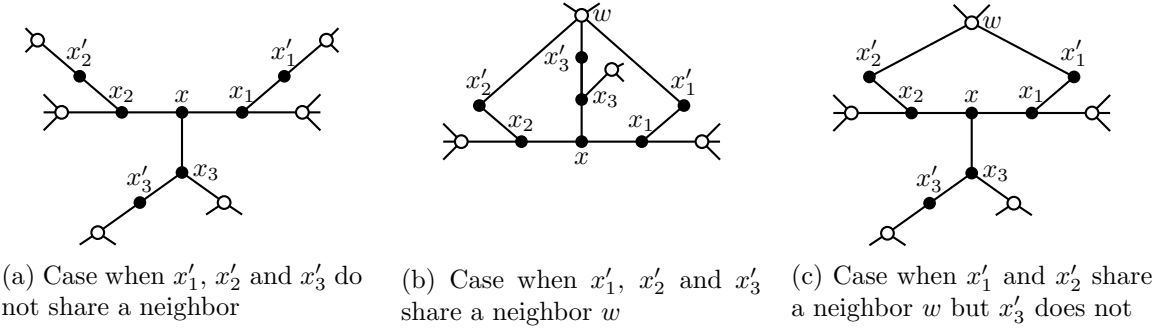%

Suppose that $x'_1$, $x'_2$, $x'_3$ share a neighbor $w$. See Figure~\ref{fig:4-3}-(b). Then
$w$ has at least three 2-neighbors and so by Lemma~\ref{4-2-nbr7}, $w$ is a  $5^+$-vertex.
Thus in the 3-dynamic coloring $\phi$ of $H$, $w$ has already at least two distinct colors in its neighbors other than the colors of $x'_1$, $x'_2$, $x'_3$. Thus eventually, the extended coloring of $G$ results that $G$ is 3-dynamically $L$-colorable, a contradiction.

Suppose that  $x'_1$ and $x'_2$ share a neighbor $w$ and $x'_3$ does not. See Figure~\ref{fig:4-3}-(c).
Then $w$ has at least two 2-neighbors  by Lemma~\ref{4-2-nbr7}, $w$ is a  $4^+$-vertex.
Thus in the 3-dynamic coloring $\phi$ of $H$, $w$ has already at least two distinct colors in its neighbors other than the colors of $x'_1$ and $x'_2$.
Thus the extended coloring of $G$ results that $G$ is 3-dynamically $L$-colorable, a contradiction.
\end{proof}

\bigskip
We use discharging technique.
We define the charge of each vertex $v$ of $G$ by its degree $\deg(v)$. Note that the average charge is less than $\frac{14}{5}$. Next, we distribute their charges by the following rules, and then we show that  the new charge of each vertex is at least $\frac{14}{5}$, which leads a contradiction.
The rules are as follows.

\bigskip

\begin{enumerate}
\item[{\bf R1.}] A  $3^+$-vertex gives $\frac{2}{5}$ to its each of 2-neighbors.
\item[{\bf R2.}] A  $3^+$-vertex  gives $\frac{1}{10}$ to its each of 3-neighbors in $W_1$.
\end{enumerate}

Let $d^*(u)$ be the new charge after discharging.  We will show  that $d^*(u) \geq  \frac{14}{5}$ for all $u\in V(G)$.
Note that by Lemma~\ref{pendent:r}-(1), each vertex of $G$ is a $2^+$-vertex.
If $\deg(u) =2$, by Lemma~\ref{pendent:r}-(2), the neighbors of $u$ are $3^+$-vertices and so it receives $\frac{2}{5}$ from each of its neighbors by {\bf R1}, which implies that
\[ d^*(u) = 2 +\frac{2}{5}+\frac{2}{5} =\frac{14}{5}.\]

If  $\deg(u) =3$, then either $u\in W_0$ or $u\in W_1$ by Lemma~\ref{4-2-nbr7}.
If  $u\in W_0$, $u$ has at most two neighbors $W_1$ by Lemma~\ref{three-W1-2-nbr7} and so $u$ gives $\frac{1}{10}$ to each of its $3$-neighbors in $W_1$ by {\bf R2} and so
\[ d^*(u) \ge 3 -2 \times \frac{1}{10} =\frac{14}{5}.\]
If   $u\in W_1$, then by Lemma~\ref{three-W1-nbr7}, the  $u$ has two $3^+$-neighbors and so it receives $\frac{1}{10}$ from each of them by {\bf R2} and so
\[ d^*(u) \ge 3 -\frac{2}{5}+\frac{1}{10}+\frac{1}{10} =\frac{14}{5}.\]
If $\deg(u) = 4$, then  by Lemma~\ref{4-2-nbr7}, $u$ has at most two 2-neighbors, and so
\[ d^*(u) \ge 4 -2\times \frac{2}{5} -2\times \frac{1}{10}=3>\frac{14}{5}.\]
If $\deg(u) \ge 5$, then
\[ d^*(u) \ge \frac{3}{5} \deg(u)>\frac{14}{5}.\]

\section{Proof of Theorem~\ref{thm:main8}}
 We use the induction on the number of vertices and the number of edges.
In the following, we let $G$ be a minimal counterexample to Theorem \ref{thm:main8}. That is, $G$ is a graph with the smallest number of vertices and edges, $mad(G)<3$, and $ch_3^d(G)\ge 9$.
Then there exists a list assignment $L$ such that $|L(v)|\ge 8$ for each $v\in V(G)$ and $G$ is not $3$-dynamically $L$-colorable.

\begin{lemma} \label{3-2-nbr8}   Any $3^{-}$-vertex has no $2$-neighbors.
\end{lemma}
\begin{proof}
Let $x$ be a $3^-$-vertex and has a 2-neighbor $y$.
Consider $H=G-xy$, deleting the edge $xy$ from $G$.
Then $mad(H)<3$.
Since $G$ is a minimal counterexample and $H$ is smaller than $G$, $H$ is 3-dynamically $L$-colorable.
Thus there is a 3-dynamic coloring $\phi$ of $H$ such that $\phi(a)\in L(a)$ for any $a\in V(H)$.
Then uncolor the vertices  $x$ and $y$.

Note that the number of forbidden colors at $x$ is at most $3+3+1=7$.
Thus $x$ has at least one available color.
We color $x$ first with that color.
Then we recolor $y$, since the number of forbidden colors at $y$ is at most $3+3=6$.
\end{proof}

\begin{lemma} \label{4-2-nbr8} For $k\in\{4,5\}$, any $k$-vertex has at most $(k-2)$ $2$-neighbors.
\end{lemma}

\begin{proof} Let $k\in\{4,5\}$ and let
$v$ be a $k$-vertex, and $v_1$, $v_2$, \ldots, $v_k$ be its neighbors.
Suppose that $v$ has at least $(k-1)$ 2-neighbors, $v_1$,  \ldots, $v_{k-1}$ .
Let $H=G-vv_1$.
Then $mad(H)<3$.
Since $G$ is a minimal counterexample and $H$ is smaller than $G$, $H$ is 3-dynamically $L$-colorable.
Thus there is a 3-dynamic 8-coloring $\phi$ of $H$ such that $\phi(a)\in L(a)$ for any $a\in V(H)$.
Then uncolor the vertex $v$ and its all 2-neighbors.

Note that the number of forbidden colors at $v$ is at most $(k-1)+3=k+2\le 7$.
Thus $v$ has at least one available color and we color $v$ first with that color.
Then we recolor each 2-neighbor of $v$ one by one.
Since the number of available colors at each 2-neighbor of $v$ is 3, and so they are colorable so that $v$ has three distinct colored neighbors.
\end{proof}

We use discharging technique.
We define the charge of each vertex $v$ of $G$ by its degree $\deg(v)$. Note that the average charge is less than $3$. Next, we distribute their charges by the following rules, and then we show that  the new charge of each vertex is at least $3$, which leads a contradiction.
The rule is as follows.

\bigskip

\begin{enumerate}
\item[{\bf R1.}] A  $4^+$-vertex gives $\frac{1}{2}$ to its each of 2-neighbors.
\end{enumerate}

Let $d^*(u)$ be the new charge after discharging. We will show  that $d^*(u) \geq 3$ for all $u\in V(G)$.
By Lemma~\ref{pendent:r}-(1), each vertex of $G$ is a $2^+$-vertex.
If $\deg(u) =2$, by Lemma~\ref{3-2-nbr8}, then the neighbors of $u$ are $4^+$-vertices and so $u$ receives $\frac{1}{2}$ from each of its neighbors by {\bf R1} and so
\[ d^*(u) = 2 +\frac{1}{2}+\frac{1}{2} =3.\]
If $\deg(u) =3$ then the charge of $u$ is not changed and so $d^*(u)=\deg(u) =3$.
If $\deg(u)=4$ then by Lemma~\ref{4-2-nbr8}, it has at most two 2-neighbors and so
 \[ d^*(u) \ge 4 -2\times \frac{1}{2}=3.\]
If $\deg(u)=5$ then by Lemma~\ref{4-2-nbr8}, it has at most three 2-neighbors and so
 \[ d^*(u) \ge 5 -3\times \frac{1}{2}>3.\]
If $\deg(u)\ge 6$, then
  \[ d^*(u) \ge \deg(v) -\deg(v)\times \frac{1}{2}\ge \frac{\deg(v)}{2}\ge3.\]

\bigskip

\noindent {\bf Acknowledgement.} We would like to thank the two anonymous reviewers for helpful and valuable comments.


\end{document}